\documentclass[review,hidelinks,onefignum,onetabnum]{siamart220329}

\usepackage{graphics} 
\usepackage{amssymb}  
\usepackage{cite}
\usepackage{amsmath}
\usepackage{algorithm}
\usepackage{algorithmic}
\usepackage{graphicx}
\usepackage{epstopdf}
\usepackage{booktabs}
\usepackage{multirow}
\usepackage{etoolbox}
\usepackage{mathdots}
\usepackage{xcolor}
\makeatletter
\patchcmd{\@makecaption}
{\scshape}
{}
{}
{}
\makeatother
\usepackage{diagbox}
\usepackage{bm}
\usepackage{subfigure}
\usepackage{lipsum}
\usepackage{braket,amsfonts,amsopn} 

\usepackage{hyperref}

\usepackage{caption}
\def\norm #1{\left\|#1\right\|}

\def\twon #1{\left\|#1\right\|_2}

\def\frobn #1{\left\|#1\right\|_{\text{F}}}

\def\abs #1{\left|#1\right|}
\def\inp #1{\left\langle#1\right\rangle}

\def\st{\text{subject to }}

\def\bC{\mathbb{C}}

\def\bR{\mathbb{R}}
\def\bE{\mathbb{E}}
\def\bP{\mathbb{P}}
\def\bT{\mathbb{T}}

\def\m #1{\boldsymbol{#1}}

\def\cA{\mathcal{A}}
\def\cC{\mathcal{C}}

\def\cH{\mathcal{H}}

\def\cL{\mathcal{L}}

\def\cN{\mathcal{N}}
\def\cP{\mathcal{P}}

\def\cT{\mathcal{T}}

\def\bee{\begin{equation}}
	\def\ene{\end{equation}}

\def\beq{\begin{eqnarray}}
	\def\enq{\end{eqnarray}}
\def\lentwo{\setlength\arraycolsep{2pt}}

\def\equ #1{\begin{equation}#1\end{equation}}
\def\equa #1{\begin{eqnarray}#1\end{eqnarray}}
\def\sbra #1{\left(#1\right)}
\def\mbra #1{\left[#1\right]}
\def\lbra #1{\left\{#1\right\}}
\def\tr #1{\text{tr}#1}

\def\rank #1{\text{rank}#1}
\def\st {\text{ subject to }}

\DeclareMathOperator*{\argmin}{arg\,min}

\usepackage{lipsum}
\usepackage{amsfonts}
\usepackage{graphicx}
\usepackage{epstopdf}
\usepackage{algorithmic}
\ifpdf
\DeclareGraphicsExtensions{.eps,.pdf,.png,.jpg}
\else
\DeclareGraphicsExtensions{.eps}
\fi

\newsiamremark{remark}{Remark}
\newsiamremark{hypothesis}{Hypothesis}
\crefname{hypothesis}{Hypothesis}{Hypotheses}
\newsiamthm{claim}{Claim}

\headers{Multichannel Frequency Estimation With Constant Amplitude }{X. Wu, Z. Yang, and Z. Xu}

\title{Multichannel Frequency Estimation with Constant Amplitude via Convex Structured Low-Rank Approximation \thanks{Submitted to the editors DATE.
		\funding{The research of the project was supported by the National Natural Science Foundation of China under Grant 61977053. (\emph{Corresponding author: Zai Yang.})}}}

\author{Xunmeng Wu \footnotemark[2] \and Zai Yang \footnotemark[2] \and Zongben Xu \thanks{School of Mathematics and Statistics, Xi’an Jiaotong University, Xi’an 710049, China. (\email{wxm1996@stu.xjtu.edu.cn}, \email{yangzai@xjtu.edu.cn}, \email{zbxu@xjtu.edu.cn}).}
	}

\usepackage{amsopn}
\DeclareMathOperator{\diag}{diag}

\ifpdf
\hypersetup{
  pdftitle={Multichannel Frequency Estimation with Constant Amplitude via Convex Structured Low-Rank Approximation},
  pdfauthor={X. Wu, Z. Yang, and Z. Xu}
}
\fi

\begin{document}

\maketitle

\begin{abstract}
	We study the problem of estimating the frequencies of several complex sinusoids with constant amplitude (CA) (also called constant modulus) from multichannel signals of their superposition. To exploit the CA property for frequency estimation in the framework of atomic norm minimization (ANM), we introduce multiple positive-semidefinite block matrices composed of Hankel and Toeplitz submatrices and formulate the ANM problem as a convex structured low-rank approximation (SLRA) problem. The proposed SLRA is a semidefinite programming and has substantial differences from existing such formulations without using the CA property. The proposed approach is termed as \textbf{S}LRA-based \textbf{A}NM for \textbf{CA} frequency estimation (\textbf{SACA}). We provide theoretical guarantees and extensive simulations that validate the advantages of SACA.
\end{abstract}

\begin{keywords}
Frequency estimation, constant amplitude/modulus, structured low-rank approximation,  Hankel-Toeplitz block matrices, convex optimization.
\end{keywords}

\begin{MSCcodes}
62M15, 15B05, 15B48, 41A29, 65F22, 93C41
\end{MSCcodes}

\section{Introduction}
Consider complex-valued multichannel signals that compose an $N\times L$ matrix whose $\sbra{j,l}$ entry is given by
\equ{
	x^\star_{j,l}  = \sum^K_{k=1} b_k e^{i(2\pi f_k (j-1)+\phi_{k,l})}, \label{eq:signal}
} 
where $f_k \in \bT \triangleq \left[-1/2, 1/2 \right)$, $b_k>0$, and $\phi_{k,l}\in \bR$ denote the $k$-th unknown normalized frequency, amplitude, and the associated phase in the $l$-th channel, respectively. Moreover, in \eqref{eq:signal} $i=\sqrt{-1}$, $K$ is the model order, $N$ is the signal length per channel and $L$ is the number of channels. It is seen that the signals among multiple channels share the same frequencies and amplitudes but have different phases. To distinguish it from general multichannel signals where the amplitudes are also different among multiple channels, we refer to \cref{eq:signal} as constant amplitude (CA) signals.

The noisy (possibly incomplete) measurements of $x^\star_{j,l}$ are obtained as
\equ{
	y_{j,l} = x^\star_{j,l} + e_{j,l}, \quad j \in \Omega, \; l = 1,\ldots,L, \label{eq:ob}
} 
where $\Omega \subseteq \lbra{1,\ldots, N}$ denotes the observation set of size $M \le N$ and $e_{j,l}$ is the noise. Note that $\Omega = \lbra{1,\ldots, N}$ and $\Omega \subset \lbra{1,\ldots, N}$ correspond to the full data case and the missing data case, respectively.

Denote $\m{a}\sbra{f} = \mbra{1,e^{i2\pi f},\ldots,e^{i2\pi \sbra{N-1} f} }^T$ and $\m{f}=\mbra{f_1,\ldots,f_K}^T$.
We rewrite \eqref{eq:ob} in matrix form as:
\equ{
	\m{Y}_{\Omega} = \m{A}_{\Omega}\sbra{\m{f}} \m{B} \m{\Phi}+\m{E}_{\Omega} = \m{X}^\star_{\Omega} + \m{E}_{\Omega}, \label{eq:signal_matrix_CA}
}
where the $K\times K$ amplitude matrix $\m{B}=\diag\sbra{\m{b}} $ with $\m{b}=\mbra{b_1,\ldots,b_K}^T$, the $K\times L$ phase matrix $\m{\Phi}$ with $\Phi_{k,l} = e^{i\phi_{k,l}}$, and $\m{A}_\Omega \sbra{\m{f}} $ is the $M \times K$ submatrix of the $N\times K$ Vandermonde matrix $\m{A} \sbra{\m{f}} =\mbra{\m{a}\sbra{f_1},\ldots,\m{a}\sbra{f_K}} $ that is composed of its rows indexed by $\Omega$. Similarly, $\m{Y}_{\Omega}$, $\m{X}^\star_{\Omega}$, and $\m{E}_{\Omega}$ are the submatrices of $\m{Y} = \mbra{y_{j,l}}$, $\m{X}^\star = [x^\star_{j,l}] = \m{A}\sbra{\m{f}} \m{B}\m{\Phi}$, and $\m{E} = \mbra{e_{j,l}}$, receptively. For notational simplicity, we write $\m{A}(\m{f})$ as $\m{A}$ hereafter.

In this paper, we are interested in estimating the frequencies $\m{f}$ from the observations $\m{Y}_\Omega$, i.e., the frequency estimation for CA signals, referred to as the CA Frequency Estimation (CAFE) problem.
CAFE is of great importance and has wide applications in array signal processing \cite{treichler1983new,gooch1986cm,shynk1996constant,wax1992unique,williams1992resolving,valaee1994alternative,leshem1999direction,shynk1996constant,leshem1999direction,van2001asymptotic,leshem2000maximum,stoica2000maximum,pesavento2023three}, radar \cite{li2007mimo,cui2013mimo}, and wireless communications \cite{liu2018toward,zhang2021overview}. For example, in array signal processing, one needs to estimate the directions of several electromagnetic sources from outputs of an antenna array. The set $\Omega$ corresponds to the antenna locations of a linear array. The full data case and missing data case arise when the antennas form a uniform linear array and a sparse linear array, respectively. Each channel in $\lbra{1,\ldots,L}$ corresponds to one temporal snapshot, and each frequency has a one-to-one mapping to the direction. The set $\lbra{b_k e^{i\phi_{k,l}}}$ in \eqref{eq:signal} corresponds to the signals emitted by the sources, which has CA (or constant modulus) among multiple snapshots. Many man-made signals, such as phase-modulated and frequency-modulated signals, exhibit such a CA property. Besides, the CA constraint is usually imposed in numerous tasks of radar and communications such as waveform design for power efficient transmission and enhanced sensing performance \cite{cui2013mimo,liu2018toward}. 

Frequency estimation is a fundamental problem in statistical signal processing and has a long history of research. It was first discussed by Prony for the noiseless case in 1795. Due to its connection to array processing, frequency estimation has been widely studied. Since World War II, the Bartlett beamformer based on Fourier analysis and the Capon beamformer have been proposed. Since the 1970s, numerous approaches have been developed, including Pisarenko’s method, subspace-based methods such as MUSIC \cite{schmidt1986multiple,barabell1983improving,stoica1995optimally} and ESPRIT \cite{roy1989esprit,park1994esprit,andersson2018esprit,yang2023nonasymptotic}, and
matrix pencil method \cite{hua1990matrix}. Readers are referred to
\cite{krim1996two,stoica2005spectral} for a review. 
Sparse optimization and compressed sensing methods have become attractive since this century, see \cite{yang2018sparse}. By exploiting the signal sparsity (i.e., $K<N$) and proposing an optimization framework for signal recovery, sparse methods do not need \emph{a priori} knowledge on the model order $K$ and are applicable to the missing data case. Rigorous theory and algorithms are established in \cite{candes2014towards,tang2013compressed,bhaskar2013atomic,yang2016exact,yang2018sample,li2015off} to deal with the continuous-valued frequencies in which atomic norm (or total variation norm) minimization (ANM) methods are proposed. Instead of directly estimating the frequency, ANM methods turn to optimize a low-rank Toeplitz matrix where the low-rankness comes from the signal sparsity. 
Structured low-rank approximation (SLRA) \cite{andersson2014new,andersson2019fixed,cai2018spectral,cai2019fast,chen2014robust,yang2023new,markovsky2008structured} is another powerful approach for frequency estimation. By the Kronecker's theorem \cite{adamyan1968infinite}, the frequency estimation is formulated into a Hankel matrix low-rank approximation problem where the rank corresponds to the model order. The Hankel matrix is also exploited in \cite{beylkin2005approximation} to develop an approximation algorithm for frequency estimation and widely used in system identification \cite{fazel2013hankel,markovsky2013structured}.
Note that all these methods focus on exploiting the Vandermonde structure of $\m{A}(\m{f})$ in signals and cannot use the CA property in CAFE concerned in the present paper.

Taking the CA property into account in frequency estimation traces back to \cite{gooch1986cm,shynk1996constant} and has been studied in the past four decades in array processing. In the language of frequency estimation, it is shown in \cite{wax1992unique,williams1992resolving,valaee1994alternative,leshem1999direction} that the exploitation of the CA property brings substantial benefits to frequency estimation including the ability of identifying $K\ge N$ frequencies and a lower Cram\'{e}r-Rao bound (CRB). But the use of the CA property results in a large number of nonconvex constraints and brings great challenges to algorithm design. To deal with it, the iterative constant modulus algorithm \cite{godard1980self,treichler1983new}, the analytic constant modulus algorithm (ACMA) \cite{van1996analytical}, and the zero-forcing variant of ACMA (ZF-ACMA) \cite{van2001asymptotic} have been proposed. These algorithms have been utilized in CAFE to estimate the Vandermonde matrix by omitting its structure, from which the frequencies are estimated using grid search or ESPRIT \cite{shynk1996constant,leshem1999direction,van2001asymptotic}.
Since the Vandermonde structure cannot be used in the first step, such two-step methods are suboptimal and suffer from the limit of $K\le N$ \cite{van1996analytical,van2001asymptotic}. A Newton scoring algorithm is proposed in \cite{leshem2000maximum} to solve the highly nonconvex maximum-likelihood estimation (MLE) problem of CAFE whose performance heavily depends on the initialization (note that the nonconvexity of the MLE comes from the Vandermonde structure besides the CA constraints). The paper \cite{stoica2000maximum} considered the case of $K=1$ and derived a simple expression for the MLE. 

Due to the great challenges brought by the CA constraints, to the best of our knowledge, few progresses have been made for CAFE in the past two decades until recently. In \cite{wu2022Direction} the authors proposed a structured matrix recovery technique (SMART) for CAFE in which the highly nonconvex MLE problem is formulated as a rank-constrained structured matrix recovery problem that is then properly solved thanks to the recent progresses on low-rank matrix recovery \cite{shen2014augmented,davenport2016overview}.  It is worth noting that all aforementioned algorithms for CAFE need the knowledge of the model order and are based on nonconvex optimization, which does not guarantee global optimality and limits their practical interest. This motivates us to develop efficient convex optimization methods to tackle CAFE with theoretical performance guarantees. 

In this paper, we propose a \textbf{S}LRA-based \textbf{A}NM approach for \textbf{CA}FE, named as \textbf{SACA}. Our main contributions are summarized below.
\begin{enumerate}
	\item To use the CA property, we propose a CA atomic norm, in which a unit-modulus constraint is imposed, and formulate CAFE in the noiseless and noisy cases as CA atomic norm minimization problems (see \Cref{sec:def}). 
	\item To make the previous CA atomic norm minimization problems be computationally tractable, we need to deal with the nonconvex CA constraints besides the infinitely many frequency variables. To this end, we show that a proper relaxation of the CA constraints in the CA atomic norm results in an equivalent atomic norm. We further cast the CA atomic norm as a convex SLRA problem by introducing multiple positive-semidefinite (PSD) Hankel-Toeplitz block matrices to capture the CA property. We show that the proposed SLRA is a semidefinite programming (SDP) and is essential to SACA by analyzing its differences from existing such formulations without using the CA property \cite{yang2016exact,yang2018sample,li2015off,fernandez2016super,steffens2018compact,li2018atomic} (see \Cref{sec:SDP}--\Cref{sec:dual}). 
	\item We develop a reasonably fast algorithm for the SLRA based on the alternating direction method of multipliers (ADMM) \cite{boyd2010distributed} and analyze its computational complexity (see \Cref{sec:dual} and \Cref{sec:ADMM}).
	\item We show theoretically the benefit of using the CA property for frequency estimation and provide theoretical guarantees for exact recovery of SACA in the noiseless case. We also derive the optimal regularization parameter in the noisy case. Extensive simulations are carried out to validate our theoretical findings and confirm the superiority of the proposed SACA (see \Cref{sec:theory} and \Cref{sec:Sim}).
\end{enumerate}

\subsection{Relations to Prior Art}
As compared to ANM for general multichannel frequency estimation  \cite{yang2016exact,yang2018sample,li2015off,fernandez2016super,steffens2018compact,li2018atomic}, the proposed SACA is different in the definition of atomic norm, computable characterization, and theoretical performance analyses. Specifically, a unit-modulus constraint is included in the definition of the set of atoms for SACA in order to exploit the CA property, which is not shared by ANM and brings challenges to the subsequent computable characterization. To make the proposed atomic norm be computable, the key is to capture the amplitudes of signals (besides the frequencies) in each channel and impose their equality by introducing only convex constraints. This challenging task is accomplished, inspired by our recent papers \cite{wu2022maximum,wu2022Direction}, by construction of a PSD block matrix composed of Hankel and Toeplitz submatrices for each channel and letting these block matrices share a same Toeplitz submatrix that is shown to capture the amplitudes and frequencies. Based on such structured matrices, a convex SLRA problem is formulated. We also show that the Hankel-Toeplitz block matrices are essential to SACA by presenting an intermediate formulation that links both the existing formulations of ANM and the proposed SLRA of SACA.

Moreover, the CA constraints make previous theoretical analyses of ANM inapplicable to SACA. For example, in existing theory for ANM with missing data, the phase vector for each sinusoidal signal is assumed to be uniformly distributed on the unit hypersphere, which cannot be satisfied for SACA where each phase vector has unit-modulus entries. In the noisy case, the key is the computation of a regularization parameter, which is expressed as the expectation of the dual atomic norm of the random noise matrix, given the analysis framework in \cite{bhaskar2013atomic}. While the dual atomic norm for ANM is expressed using the $\ell_2$-norm, it is cast by the $\ell_1$-norm for SACA that is more difficult to deal with. The aforementioned differences inevitably bring new challenges to our theoretical performance analyses of SACA.

The SLRA problem of SACA is inspired by our recent papers \cite{wu2022maximum,wu2022Direction}, as mentioned previously. In contrast to SMART \cite{wu2022Direction} in which nonconvex rank constraints are imposed to explicitly restrict the model order, the signal sparsity in SACA is promoted by minimization of the matrix trace norm. From this point of view, SACA is a convex relaxation of SMART. As compared to SMART, SACA can always find the global optimal solution and has theoretical performance guarantees. Another advantage of SACA is that it does not require the model order.

Compared to the structured matrix embedding and recovery (StruMER) method in \cite{wu2023multichannel}, SACA is different in the estimation problem, optimization framework, and optimization model. Firstly, StruMER was developed for frequency estimation for general (not necessarily CA) multichannel signals, while SACA is proposed for CA multichannel signals. Secondly, StruMER is a nonconvex optimization approach based on maximum likelihood estimation, while SACA is a convex optimization method based on atomic norm minimization. Thirdly, the optimization model of StruMER is a rank-constrained problem, while that of SACA is semidefinite programming (SDP). Although both models involve multiple Hankel-Toeplitz block matrices, which is a similarity between StruMER and SACA, they impose different linear constraints on the Toeplitz submatrices to capture distinct signal structures.

Partial results of this work, as a 5-page short paper, have been submitted to the upcoming conference \cite{wu2024direction} during the review of this paper, in which only the optimization formulation of SACA with simplified derivations and partial simulation results in the noiseless case are included. 

\subsection{Notation}
Boldface letters are reserved for vectors and matrices. The sets of real and complex numbers are denoted $\bR$ and $\bC$, respectively. For vector $\m{x}$, $\m{x}^T$, $\overline{\m{x}}$, $\m{x}^H$, $\norm{\m{x}}_1$,  $\norm{\m{x}}_2$, and $\norm{\m{x}}_{\infty}$ denote its transpose, complex conjugate, conjugate transpose, $\ell_1$-norm, $\ell_2$-norm, and $\ell_\infty$-norm, respectively. For matrix $\m{X}$, its transpose, complex conjugate, conjugate transpose,  pseudo-inverse, Frobenius norm, rank, trace, and column space are denoted $\m{X}^T$, $\overline{\m{X}}$, $\m{X}^H$, $\m{X}^\dagger$, $\frobn{\m{X}}$, $\rank\sbra{\m{X}}$, $\tr\sbra{\m{X}}$, and $\text{range}\sbra{\m{X}}$, respectively. The inner product is represented by  $\langle \cdot,\cdot \rangle_{\bR} $. $\m{X} \succ \m{0}$ and $\m{X} \succeq \m{0}$ mean that $\m{X}$ is Hermitian positive definite and  PSD, respectively. The notation $\abs{\cdot}$ denotes the amplitude of a scalar. $\lceil x \rceil$ (or $\lfloor x \rfloor$) denotes the smallest (or largest) integer greater (or less) than or equal to $x$. The diagonal matrix with vector $\m{x}$ on the diagonal is denoted $\diag\sbra{\m{x}}$. An identity matrix is denoted as $\m{I}$. The $j$-th entry of vector $\m{x}$ is $x_j$, the $\sbra{i,j}$-th entry of $\m{X}$ is $x_{i,j}$, and the $i$-th row (or $j$-th column) of $\m{X}$ is $\m{X}_{i,:}$ (or $\m{X}_{:,j}$). Denote $\m{X}_\Omega$ as the submatrix of $\m{X}$ formed by its rows indexed by the subset $\Omega$.
Denote the complementary set of $\Omega$ as $\Omega^c = \lbra{1,\ldots,N} - \Omega$.
Given vectors $\m{x} \in\bC^{2n-1}$ and $\m{t} \in\bC^{2n-1}$ satisfying $t_j = \overline{t}_{2n-j}$ for $j=1,\ldots,n$,
\equ{
	\cH\m{x} = \begin{bmatrix} x_1 & x_2 & \dots & x_{n} \\ x_2 & x_3 & \cdots & x_{n+1} \\ \vdots & \vdots & \ddots & \vdots \\ x_n & x_{n+1} & \dots & x_{2n-1} \end{bmatrix} \quad \text{and} \quad  \cT\m{t} = \begin{bmatrix} t_n & t_{n+1} & \dots & t_{2n-1} \\ t_{n-1} & t_n & \cdots & t_{2n-2} \\ \vdots & \vdots & \ddots & \vdots \\ t_1 & t_2 & \dots & t_n \end{bmatrix} 
}
denote an $n \times n$ Hankel matrix $\cH\m{x}$ and an $n \times n$ Hermitian-Toeplitz matrix, respectively. Let $\cH^H$ and $\cT^H$ denote the adjoint of the Hankel operator $\cH$ satisfying $\langle \cH\m{x}, \m{X} \rangle = \langle \m{x}, \cH^H\m{X} \rangle$ for any $n \times n$ matrix $\m{X}$ and the adjoint of the Toeplitz operator $\cT$ satisfying $\langle \cT\m{t}, \m{X} \rangle = \langle \m{t}, \cT^H\m{X} \rangle$ where $\m{X}$ is Hermitian, respectively.

\section{SLRA-based ANM for CAFE (SACA)} \label{sec:ANM-CA}
\subsection{Definition of ANM for CA Signals} \label{sec:def}
By utilizing specific structures of a signal, atomic norm \cite{chandrasekaran2012convex} offers a generic method for finding a sparse representation of the signal. With appropriately chosen atoms, the atomic norm generalizes both the $\ell_1$-norm for sparse signal recovery and the nuclear norm for low-rank matrix recovery.

To find a sparse representation of signal $\m{X}^\star$ for CAFE, we define the set of atoms
\equ{
	\cA \triangleq \lbra{ \m{a}(f) \m{\psi}: f \in \bT, \m{\psi} \in \bC^{1 \times  L}, \abs{\psi_l}=1, l=1,\ldots,L }, \label{eq:A0}
}
as building blocks of $\m{X}^\star$ in which the CA structure is enforced explicitly.
To exploit the signal sparsity, the CA atomic norm of $\m{X} \in \bC^{N\times L}$ is defined as the gauge function of the convex hull of $\cA$ \cite{chandrasekaran2012convex}:
\equ{
	\begin{split}
		\norm{\m{X}}_{\cA} & \triangleq \inf \lbra{ \lambda>0: \m{X}\in \lambda \; \text{conv}\sbra{\cA} } \\
		& = \inf \lbra{ \sum_k b_k : \m{X} = \sum_{k} b_k \m{a}_k, \m{a}_k \in \cA, b_k \ge 0  }.
	\end{split} \label{eq:atomic1}
}

Following the literature of frequency estimation based on ANM \cite{bhaskar2013atomic,tang2013compressed,yang2016exact,yang2018sample,li2015off,li2018atomic}, in the absence of noise, we consider the following CA atomic norm minimization problem:
\equ{
	\min_{\m{X}} \norm{\m{X}}_{\cA}, \st \m{X}_{\Omega} = \m{X}^\star_{\Omega}, \label{eq:p1}
}
to recover the noiseless full signal matrix $\m{X}^\star$ and its frequencies $\m{f}$. When the observations are corrupted by i.i.d. zero-mean Gaussian noise $\lbra{e_{j,l}}$, we consider the following CA atomic norm denoising problem:
\equ{
	\min_{\m{X}} \frac{1}{2} \frobn{\m{Y}_{\Omega}-\m{X}_{\Omega}}^2 + \tau \norm{\m{X}}_{\cA}, \label{eq:atomic_min2}
}
where $\tau$ is a regularization parameter that will be specified in \Cref{sec:tau}.

\subsection{Computable Characterization via SLRA} \label{sec:SDP}
Though the problems in \eqref{eq:p1} and \eqref{eq:atomic_min2} are convex, they are semi-infinite programs with infinitely many variables and cannot be practically solved. To solve them, we need a computable characterization for $\norm{\m{X}}_{\cA}$, which is challenging due to the nonconvex CA constraints $\abs{\psi_l}=1,l=1,\ldots,L$ on $\cA$ in \eqref{eq:A0}. To this end, we define a new set of atoms
\equ{
	\cA' \triangleq \lbra{ \m{a}(f) \m{\psi}: f \in \bT, \m{\psi} \in \bC^{1 \times  L}, \norm{\m{\psi}}_{\infty} = 1}. 
}
Define $\norm{\m{X}}_{\cA'}$ as that in \eqref{eq:atomic1} by replacing $\cA$ with $\cA'$.
We have the following result showing that the relaxation from $\cA$ to $\cA'$ is tight in the convex optimization framework.
\begin{lemma} \label{lem:twonorm}
	The two atomic norms associated with $\cA'$ and $\cA$ are equivalent, i.e., 
	\equ{
		\norm{\m{X}}_{\cA'} = \norm{\m{X}}_{\cA}. \label{eq:=}
	} 
\end{lemma}
\begin{proof}
	See \cref{append:1}.
\end{proof}

SLRA usually refers to the problem of approximation of a given data matrix by another structured low-rank matrix, which provides a tool for fitting data by a low complexity model \cite{chu2003structured,markovsky2008structured,markovsky2013structured}. Various matrix structures including Hankel, Toeplitz, and Sylvester have been exploited in applications according to the specific structures of signal models. The low-rankness arises from the low complexity, such as sparsity, of the signal model. 

Inspired by SLRA and the construction of Hankel-Toeplitz matrices in \cite{wu2022maximum,wu2022Direction}, we show that $\norm{\m{X}}_{\cA'}$ can be cast as a convex SLRA problem. It can be regarded as a convex relaxation of the classical SLRA problem by swapping the approximation criterion and the low-rank constraint and using the trace norm minimization to promote the low-rankness. Formally, we have the following theorem, of which the proof will be deferred to the subsequent subsection.

\begin{theorem}  \label{thm:SLRA}
	For any $\m{X}\in \bC^{N\times L}$, let $\text{SLRA} \sbra{\m{X},n}$ denote the optimal value of the following convex SLRA problem:
	\equ{ 
		\begin{split}
			& \min_{\m{t}\in \bC^{2n-1},\m{Z}\in \bC^{\sbra{2n-1} \times L}} \frac{1}{n} \tr\sbra{\cT\m{t}}, \\
			&  \st 
			 \begin{bmatrix} \cT\overline{\m{t}} & \cH\overline{\m{Z}}_{:,l} \\ \cH\m{Z}_{:,l} & \cT\m{t} \end{bmatrix} \succeq \m{0}, \; l = 1,\ldots, L, \\
			& \qquad \qquad \quad \m{Z}_{\lbra{1,\ldots,N}} = \m{X}, \label{eq:SDP}
		\end{split}
	}
	where $n \ge \lceil \sbra{N+1}/2 \rceil$, $\cT\m{t} $ and $\cH\m{Z}_{:,l}$ are $n \times n$ Hermitian-Toeplitz and Hankel matrices, respectively, and $\m{Z}_{\lbra{1,\ldots,N}}$ is the submatrix of $\m{Z}$ formed by its first $N$ rows. Then, the following statements hold true:
	\begin{enumerate}
		\item $\text{SLRA} \sbra{\m{X},n}$ is monotonic non-decreasing with respect to $n$, i.e., 
		\equ{
			\text{SLRA} \sbra{\m{X},n} \le \text{SLRA} \sbra{\m{X},n+1}; \label{eq:nonde}
		}
		\item For any $n \ge \lceil \sbra{N+1}/2 \rceil$, we have  
		\equ{
			\text{SLRA} \sbra{\m{X},n} \le \norm{\m{X}}_{\cA'}; \label{eq:le}
		}
		\item If the optimal solution to matrix $\cT\m{t}$ in \eqref{eq:SDP} is rank-deficient, then we further have that
		\equ{
			\text{SLRA} \sbra{\m{X},n} = \norm{\m{X}}_{\cA'}.
		}
	\end{enumerate}
\end{theorem}

\Cref{thm:SLRA} shows that the SLRA in \eqref{eq:SDP} provides a lower bound for the atomic norm $\norm{\m{X}}_{\cA'}$ and is an exact characterization of $\norm{\m{X}}_{\cA'}$ when the optimal solution to $\cT\m{t}$ is rank-deficient. Interestingly, note that the trace norm in the objective of the SLRA promotes low-rankness of the solution to $\cT\m{t}$. Therefore, if $n$ is chosen large enough, then it is expected that the solution to $\cT\m{t}$ becomes rank-deficient. In fact, it is found empirically that $n = \lceil \sbra{N+1}/2 \rceil$ is usually sufficient if $K < n$.

It is worth noting that by the proof of \Cref{thm:SLRA} in \Cref{sec:proof}, the rank-deficient optimal solution $\cT\m{t}^*$ is given by \eqref{eq:T2}, where $\m{X}=\m{A}\m{B}^*\m{\Psi}^*$ is the atomic decomposition achieving the atomic norm. Evidently, $\cT\m{t}^*$ captures all the information of the amplitudes and the frequencies. Therefore, it is seen from \eqref{eq:SDP} that we have constructed a number $L$ of PSD Hankel-Toeplitz block matrices for $L$ channels. By letting the $L$ block matrices share the same Toeplitz submatrix, we have successfully imposed the constraints that all the channels share the same amplitudes and frequencies. We show an example of $L=3$ Hankel-Toeplitz block matrices in \cref{fig:HT}.

\begin{figure*}[htb]
	\centering
	\includegraphics[width=10cm]{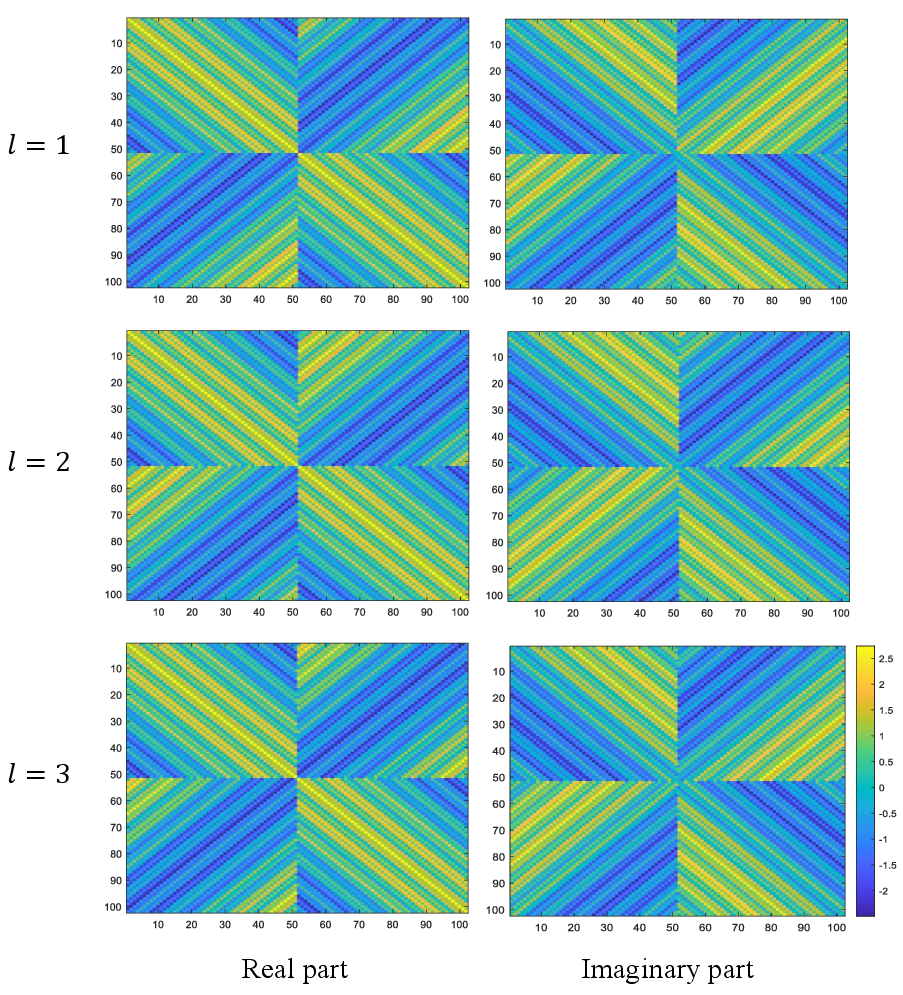}
	\caption{An example of a number $L=3$ of Hankel-Toeplitz block matrices $\lbra{\protect \begin{bmatrix} \cT\overline{\m{t}} & \cH\overline{\m{X}}_{:,l} \protect \\ \cH\m{X}_{:,l} & \cT\m{t} \protect \end{bmatrix},l=1,2,3}$, where the CA signals $\m{X}=\m{A}\sbra{\m{f}}\m{B}\m{\Phi} \in \bC^{N\times L}$ and $\cT\m{t} = \m{A}_n\sbra{\m{f}} \m{B} \m{A}^H_n\sbra{\m{f}} \in \bC^{n\times n}$ (the definition of $\m{A}_n\sbra{\m{f}}$ is given in \cref{lem:T}) with $N=101$, $n=51$, $\m{f}=[-0.1,0.01,0.35]^T$, and $\m{B} = \diag\sbra{ [0.71,1.19,0.84]}$.} \label{fig:HT}
\end{figure*} 

Making use of \eqref{eq:=} and \cref{thm:SLRA}, we propose the following convex SLRA problems:
\equ{ 
	\begin{split} 
		& \min_{\m{t},\m{Z}}  \frac{1}{n} \tr\sbra{\cT\m{t}},  \\
		& \st  \begin{bmatrix} \cT\overline{\m{t}} & \cH\overline{\m{Z}}_{:,l} \\ \cH\m{Z}_{:,l} & \cT\m{t} \end{bmatrix} \succeq \m{0}, \; l = 1,\ldots, L, \\
		& \qquad \qquad \quad \m{Z}_{\Omega} = \m{X}^\star_{\Omega}, \label{eq:p_SDP}
	\end{split}
}
and
\equ{ 
	\begin{split} 
		& \min_{\m{t},\m{Z}} \frac{1}{2}\frobn{\m{Y}_{\Omega}-\m{Z}_{\Omega}}^2 +  \frac{\tau}{n} \tr\sbra{\cT\m{t}}, \\
		& \st \begin{bmatrix} \cT\overline{\m{t}} & \cH\overline{\m{Z}}_{:,l} \\ \cH\m{Z}_{:,l} & \cT\m{t} \end{bmatrix} \succeq \m{0}, \; l = 1,\ldots, L, \label{eq:p_SDP_completion}
	\end{split}
}
corresponding to the noiseless completion problem in \eqref{eq:p1} and the denoising problem in \eqref{eq:atomic_min2}, respectively. 

Both the SLRA problems in \eqref{eq:p_SDP} and \eqref{eq:p_SDP_completion} are SDP, which can be solved using an off-the-shelf SDP solver such as SDPT3 \cite{toh2012implementation}. Given the optimal solution $(\m{t}^*,\m{Z}^*)$, the frequencies can be obtained by finding the Vandermonde decomposition of $\cT\m{t}^*$ in \eqref{eq:T2} via a subspace-based method such as root-MUSIC \cite{barabell1983improving}, and then the amplitudes and phases can be retrieved from $\m{Z}^*$.

\subsection{Proof of \Cref{thm:SLRA}} \label{sec:proof}
The following lemmas, including the classical Carath\'{e}odory-Fej\'{e}r's Theorem for Toeplitz matrices \cite[Theorem 11.5]{yang2018sparse} and a lemma regarding factorization of Hankel matrices \cite[Lemma 4]{yang2016vandermonde}, will play key roles in our proof.
\begin{lemma}(\cite[Theorem 11.5]{yang2018sparse}) \label{lem:T}
	Any PSD Toeplitz matrix $\cT\m{t}\in \bC^{n\times n}$ of rank $K<n$ admits the unique Vandermonde decomposition $\cT\m{t} = \m{A}_n(\m{f}) \diag\sbra{\m{p}} \m{A}^H_n(\m{f})$ where $\lbra{f_k}$ are distinct, $p_k>0,k=1,\ldots,K$, and $\m{A}_n(\m{f}) = \mbra{\m{a}_n\sbra{f_1},\ldots,\m{a}_n\sbra{f_K}}$ is an $n\times K$ Vandermonde matrix with $\m{a}_n\sbra{f_k} = \mbra{1,e^{i2\pi f_k},\ldots,e^{i2\pi f_k\sbra{n-1}}}^T$.
\end{lemma}
\begin{lemma}(\cite[Lemma 4]{yang2016vandermonde}) \label{lem:H}
	If a Hankel matrix $\cH\m{x}$ can be factorized as $\cH\m{x} = \m{A}_n(\m{f}) \m{G}\m{A}^T_n(\m{f})$ where $\lbra{f_k}$ are distinct and $\m{G}\in \bC^{K \times K}$, $K < n$, then $\m{G}$ must be a diagonal matrix.
\end{lemma}

Now we begin to prove \Cref{thm:SLRA}. To show the first statement, note that the SLRA problem in \eqref{eq:SDP} with the dimension parameter $n$ shares the same objective function as that with $n+1$, but the Hankel-Toeplitz matrices in the two SLRA problems have different sizes. For each $l$, the Hankel-Toeplitz matrix in the SLRA problem with $n$ is a principal submatrix of that in the one with $n+1$. As compared to the SLRA problem with $n$, the one with $n+1$ has stronger PSD constraints and thus a smaller feasible domain, resulting in an increase in the optimal value of the minimization problem. Consequently, we have \eqref{eq:nonde}.

To show the second statement, for any atomic decomposition of $\m{X}$ given by:
\equ{
	\m{X} = \sum^{K'}_{k=1} \m{a}(f_k) b_k 
	\m{\Psi}_{k,:}, 
}
where $b_k > 0$, $\m{\Psi}_{k,:} \in \bC^{1\times L}$ and $\norm{\m{\Psi}_{k,:}}_{\infty} = 1$, we let $\m{Z} = \sum^{K'}_{k=1} \m{a}_{2n-1}(f_k) b_k \m{\Psi}_{k,:}$, where $\m{a}_{2n-1}\sbra{f_k} = \mbra{1,e^{i2\pi f_k},\ldots,e^{i2\pi f_k\sbra{2n-2}}}^T$, and thus $\m{Z}_{\lbra{1,\ldots,N}}=\m{X}$. It can easily be shown that
\equ{
	\cH\m{Z}_{:,l} = \m{A}_n \diag\sbra{\m{B}\m{\Psi}_{:,l}} \m{A}_n^T,
}
where $\m{B}=\diag\sbra{\mbra{b_1,\ldots,b_{K'}}}$ and $\m{A}_n = \m{A}_n(\m{f}) $ is of dimension $n\times K'$.
Let $\cT \m{t}$ be 
\equ{
	\cT \m{t} = \m{A}_n \m{B} \m{A}_n^H. \label{eq:T}
} 
It follows that for each $l=1,\ldots,L$,
\equ{ \label{eq:tmp1}
	\begin{split}
		\begin{bmatrix} \cT\overline{\m{t}} & \cH\overline{\m{Z}_{:,l}} \\ \cH\m{Z}_{:,l} & \cT\m{t} \end{bmatrix} & = \begin{bmatrix} \overline{\m{A}_n}\m{B} \m{A}^T & \overline{\m{A}_n} \diag\sbra{\m{B}\overline{\m{\Psi}}_{:,l}} \m{A}_n^H \\  \m{A}_n \diag\sbra{\m{B}\m{\Psi}_{:,l}} \m{A}_n^T & \m{A}_n\m{B} \m{A}_n^H \end{bmatrix} \\
		& = \begin{bmatrix} \overline{\m{A}_n} & \m{0} \\ \m{0} & \m{A}_n
		\end{bmatrix} \begin{bmatrix} \m{B} & \diag\sbra{\m{B}\overline{\m{\Psi}}_{:,l}} \\ \diag\sbra{\m{B}\m{\Psi}_{:,l}}  & \m{B} \end{bmatrix}
		\begin{bmatrix} \overline{\m{A}_n} & \m{0} \\ \m{0} & \m{A}_n \end{bmatrix}^H.
	\end{split}	
}
	We have the factorization
	\equ{ \label{eq:tmp2}
		\begin{split}
			& \begin{bmatrix} \m{B} & \diag\sbra{\m{B}\overline{\m{\Psi}}_{:,l}} \\ \diag\sbra{\m{B}\m{\Psi}_{:,l}}  & \m{B} \end{bmatrix} \\
			& = \begin{bmatrix} \m{I} &  \diag\sbra{\m{B}\overline{\m{\Psi}}_{:,l}} \m{B}^{-1} \\ \m{0} & \m{I} \end{bmatrix}  \cdot \begin{bmatrix} \m{B}-\diag\sbra{\m{B}\overline{\m{\Psi}}_{:,l}} \m{B}^{-1}\diag\sbra{\m{B}\m{\Psi}_{:,l}} & \m{0} \\ \m{0} & \m{B} \end{bmatrix} \\
			& \qquad \cdot \begin{bmatrix} \m{I} &  \diag\sbra{\m{B}\overline{\m{\Psi}}_{:,l}} \m{B}^{-1} \\ \m{0} & \m{I} \end{bmatrix}^H,
		\end{split}
	}
	where the Schur complement of $\m{B}$ satisfies
	\equ{
		\m{B}-\diag\sbra{\m{B}\overline{\m{\Psi}}_{:,l}} \m{B}^{-1}\diag\sbra{\m{B}\m{\Psi}_{:,l}} = \m{0}.
	}
	Since $\m{B}\succ \m{0}$, the matrix in \cref{eq:tmp2} is PSD and thus the matrix in \cref{eq:tmp1} is PSD for each $l=1,\ldots,L$.
Consequently, we have constructed a feasible solution $\sbra{\m{t},\m{Z}}$ to the SLRA problem in \eqref{eq:SDP}, at which the objective function equals $\frac{1}{n} \tr\sbra{\cT\m{t}} = \sum^{K'}_{k=1} b_k$. It follows that the optimal value $\text{SLRA}(\m{X},n) \le \sum^{K'}_{k=1} b_k$. Since the inequality holds for any atomic decomposition of $\m{X}$, we have that $\text{SLRA}(\m{X},n) \le \norm{\m{X}}_{\cA'}$ by the definition of the atomic norm.

To show the third statement, we suppose that $(\m{t}^*,\m{Z}^*)$ is an optimal solution to the SLRA problem in \eqref{eq:SDP}, with $\rank \sbra{\cT\m{t}^*}=r<n$. It follows from \eqref{eq:SDP} and the column inclusion property of PSD block matrices \cite[Observation 7.1.10]{horn2012matrix}\cite[Proposition 2]{horn2020rank} that
\equ{
	\cT\m{t}^* \succeq \m{0} \quad \text{and} \quad \cH\m{Z}^*_{:,l} \in \text{range}\sbra{\cT\m{t}^*}, \; l=1,\ldots,L. \label{eq:THconst}
}
Using \cref{lem:T}, there exist distinct $\lbra{f^*_k}_{k=1}^{r}$ and $\lbra{b^*_k>0}_{k=1}^{r}$ such that $\cT\m{t}^*$ admits the Vandermonde decomposition 
\equ{
	\cT \m{t}^* = \m{A}_n \m{B}^* \m{A}_n^H, \label{eq:T2}
} 
where $\m{A}_n$ is redefined as an $n\times r$ Vandermonde matrix with respect to $\lbra{f^*_k}$ and $\m{B}^*=\diag\sbra{\mbra{b^*_1,\ldots,b^*_{r}}}$.
It then follows from \eqref{eq:THconst} that there exist $r\times n$ matrices $\lbra{\m{G}^l}$ such that for each $l=1,\ldots,L$,
\equ{
	\cH\m{Z}^*_{:,l} = \m{A}_n \m{G}^l = \sbra{\m{G}^l}^T  \m{A}_n^T,
}
where the second equality follows from the symmetry of $\cH\m{Z}^*_{:,l}$. We further have $\m{G}^l=\m{A}_n^\dagger\sbra{\m{G}^l}^T\m{A}_n^T$ and thus
\equ{
	\cH\m{Z}^*_{:,l} = \m{A}_n \widetilde{\m{S}}^l \m{A}_n^T, \label{eq:HVandec}
}
where $\widetilde{\m{S}}^l = \m{A}_n^\dagger\sbra{\m{G}^l}^T$ is an $K\times K$ matrix. Given the factorization in  \eqref{eq:HVandec} with $r<n$, it follows from \cref{lem:H} that $\widetilde{\m{S}}^l$ must be diagonal, i.e., $\widetilde{\m{S}}^l = \diag\sbra{\m{S}_{:,l}} $ where $\m{S} \in \bC^{r \times L}$, so that \eqref{eq:HVandec} is a Vandermonde decomposition.
Applying \eqref{eq:T}, \eqref{eq:HVandec} and the PSD constraints in \eqref{eq:SDP}, we have that
\equ{
	\begin{split}
		\begin{bmatrix} \cT\overline{\m{t}^*} & \cH\overline{\m{Z}^*}_{:,l} \\ \cH\m{Z}^*_{:,l} & \cT\m{t}^* \end{bmatrix}  = \begin{bmatrix} \overline{\m{A}_n} & \m{0} \\ \m{0} & \m{A}_n
		\end{bmatrix} \begin{bmatrix} \m{B}^* & \diag\sbra{\overline{\m{S}}_{:,l}} \\ \diag\sbra{\m{S}_{:,l}}  & \m{B}^* \end{bmatrix}
		\begin{bmatrix} \overline{\m{A}_n} & \m{0} \\ \m{0} & \m{A}_n \end{bmatrix}^H \succeq \m{0},
	\end{split}	\nonumber
}
and thus $\begin{bmatrix} \m{B}^* & \diag\sbra{\overline{\m{S}}_{:,l}} \\ \diag\sbra{\m{S}_{:,l}}  & \m{B}^* \end{bmatrix} \succeq \m{0}$.
It follows from \cref{eq:tmp2} that this is true only if the Schur complement 
\equ{
	\m{B}^* - \diag\sbra{\overline{\m{S}}_{:,l}} \sbra{\m{B}^*}^{-1} \diag\sbra{\m{S}_{:,l}} \succeq \m{0},
}
or equivalently,
\equ{
	\abs{S_{k,l}} \le b^*_k, \; k = 1,\ldots,r.
} 

Since $\frac{1}{n} \tr\sbra{\cT\m{t}^*} = \sum^{r}_{k=1} b^*_k$ is minimized in the objective, we must have 
\equ{
	b^*_k = \max_{l} \abs{S_{k,l}}.
} 
Then we have $S_{k,l} = b^*_k \psi^*_{k,l}$ with $\abs{\psi^*_{k,l}}\le 1, \; k = 1,\ldots,r, \; l=1,\ldots,L$. It follows from \eqref{eq:HVandec} that $\m{Z}^* = \sum_{k=1}^{r} \m{a}_{2n-1}\sbra{f^*_k}b^*_k \m{\Psi}^*_{k,:}$ with $\norm{\m{\Psi}^*_{k,:}}_{\infty} = 1$ and then $\m{X}=\m{Z}^*_{\lbra{1,\ldots,N}}=\sum^{r}_{k=1} \m{a}(f^*_k) b^*_k \m{\Psi}^*_{k,:}$. It follows that $\text{SLRA}(\m{X},n) = \sum_{k=1}^{r} b^*_k \ge \norm{\m{X}}_{\cA'}$. Combining this inequality and the result in the second statement, we conclude that $\text{SLRA}(\m{X},n) = \norm{\m{X}}_{\cA'}$, completing the proof.

\subsection{Why the Proposed SLRA in \eqref{eq:SDP} is Essential to SACA?} \label{sec:essential}
The computable characterizations of existing ANM methods \cite{tang2013compressed,bhaskar2013atomic,yang2016exact,yang2018sample,li2015off,fernandez2016super,steffens2018compact,li2018atomic} are based on a PSD Toeplitz-only block matrix. In particular, the characterization of ANM for general multichannel signals \cite{yang2016exact,li2015off} is given by:
\equ{
	\begin{split}
		\min_{\m{t}\in \bC^{2N-1},\m{W}\in \bC^{L\times L}} \frac{1}{2\sqrt{N}}\sbra{\tr\sbra{\cT\m{t}}+\tr\sbra{\m{W}}},  \st \begin{bmatrix} \m{W} & \m{X}^H \\ \m{X} & \cT\m{t}  \end{bmatrix} \succeq \m{0}.
	\end{split} \label{eq:SDP_T}
}
This formulation constrains the same frequencies among the channels, but it does not utilize the CA property and thus cannot characterize the CA atomic norm in \eqref{eq:atomic1} concerned in this paper.

As compared to \eqref{eq:SDP_T}, the proposed SLRA problem in \eqref{eq:SDP} for SACA have two main differences: 1) a number $L$ of PSD block matrices are constructed in \eqref{eq:SDP}, separately for each channel, that share some common submatrices, while a single PSD block matrix is used for all channels in \eqref{eq:SDP_T}, and 2) the Hankel-Toeplitz structured block matrices are used in \eqref{eq:SDP} instead of the Toeplitz-only structure in \eqref{eq:SDP_T}. These two differences ensure that our SLRA problem in \eqref{eq:SDP} can  exploit the CA property and characterize the CA atomic norm. As an interesting consequence, to be empirically shown in \Cref{sec:Sim}, a number $K\ge N$ of frequencies can be estimated in CAFE by using the SLRA problem in \eqref{eq:SDP}, which is impossible by using that in \eqref{eq:SDP_T}.

To explore how the aforementioned two differences influence the performance, we consider the following intermediate formulation between \eqref{eq:SDP} and \eqref{eq:SDP_T}:
\equ{
	\begin{split}
		\min_{\m{t}\in \bC^{2N-1},w\in \bC} \frac{1}{2N} \tr\sbra{\cT\m{t}} + \frac{1}{2}w,  \st \begin{bmatrix} w & \m{X}_{:,l}^H \\ \m{X}_{:,l} & \cT\m{t}  \end{bmatrix} \succeq \m{0}, \; l=1,\ldots,L,
	\end{split} \label{eq:SDP_T_var}
}
in which we have constructed the block matrices separately for each channel but used the Toeplitz-only rather than the Hankel-Toeplitz structure. One might expect that the shared Toeplitz matrix $\cT\m{t}$ in \eqref{eq:SDP_T_var} could capture the amplitudes and frequencies information and fully exploit the CA property, as in \eqref{eq:SDP}. But it will shown via numerical results in \Cref{sec:Sim} that, unlike \eqref{eq:SDP}, the formulation in \eqref{eq:SDP_T_var} can only use the CA property to some extent. This implies that the first difference between \eqref{eq:SDP} and \eqref{eq:SDP_T} only enables us to partially exploit the CA property, and the Hankel-Toeplitz matrices in the second difference are essential for us to make full use of it.

\subsection{Duality} \label{sec:dual}
For a matrix $\m{V}\in \bC^{N\times L}$, the dual norm of the CA atomic norm $\norm{\cdot}_{\cA}$ is given by:
\equ{
	\begin{split}
		\norm{\m{V}}_{\cA}^* 
		& = \sup_{\norm{\m{X}}_{\cA}\le 1} \langle \m{V}, \m{X} \rangle_{\bR} = \sup_{\m{a}\sbra{f}\m{\psi}\in \cA} \langle \m{V}, \m{a}\sbra{f}\m{\psi} \rangle_{\bR} \\
		& = \sup_{ f \in \bT, \abs{\psi_l} = 1 } \langle \m{a}\sbra{f}^H\m{V}, \m{\psi}\rangle_{\bR} = \sup_{f \in \bT} \norm{\m{a}\sbra{f}^H\m{V}}_1,
	\end{split} \label{eq:dualnorm}
}
where the last equality holds since the supremum is achieved if $\psi_l$ takes the sign of $\m{a}\sbra{f}^H\m{V}_{:,l},l=1,\ldots,L$. 
The dual problem of \eqref{eq:p1} is then given by:
\equ{
	\max_{\m{V}} \left \langle \m{V}_{\Omega}, \m{X}_{\Omega}^\star \right \rangle_{\bR}, \st \norm{\m{V}}^*_{\cA} \le 1 \; {\rm and} \; \m{V}_{\Omega^c} = \m{0} \label{eq:p1_dual}
}
following from a standard Lagrangian analysis \cite{boyd2004convex}. The dual problem of \eqref{eq:atomic_min2} can be derived similarly.

It is possible to derive a computable characterization for the dual problem in \eqref{eq:p1_dual} by casting the constraint $\norm{\m{V}}^*_{\cA}\le 1$ as linear matrix inequalities (LMIs) using theory of positive trigonometric polynomials; see related derivations in \cite{dumitrescu2007positive,candes2014towards,fernandez2016super}. Unfortunately, it is unclear how to do that in our case due to the inclusion of the $\ell_1$-norm (see \eqref{eq:dualnorm}). Instead, we have derived the SLRA problem in \eqref{eq:p_SDP} for the primal problem in \eqref{eq:p1}. Interestingly, the dual problem of our SLRA problem in \eqref{eq:p_SDP} is given by (see derivations in \cref{append:dual}):
\begin{subequations} \label{eq:dualSDP}
	\begin{align}
		\max_{\m{V},\lbra{\m{U}^l},\lbra{\m{W}^{l,1}},\lbra{\m{W}^{l,2}}}  & \left \langle \m{V}_{ \Omega}, \m{X}^{\star}_{\Omega} \right \rangle_{\bR}, \nonumber \\
		\st & \begin{bmatrix} \m{W}^{l,1} & \sbra{\m{U}^l}^H \\ \m{U}^l & \m{W}^{l,2} \end{bmatrix} \succeq \m{0},  \; \cT^H\lbra{ \sum^L_{l=1}\sbra{\overline{\m{W}^{l,1}}+\m{W}^{l,2}} } =  \m{\xi},  \label{eq:dp2} \\
		&\m{V}_{:,l} = -2 \sbra{\cH^H\m{U}^l}_{\lbra{1,\ldots,N}},  \; \sbra{\cH^H\m{U}^l}_{\lbra{N,\ldots,2n-1}} = \m{0}, \ l = 1,\ldots, L,  \label{eq:dp4} \\
		&\m{V}_{\Omega^c} = \m{0}, 
	\end{align} 	
\end{subequations}where $\m{V}\in \bC^{N\times L}$, $\m{U}^l,\m{W}^{l,1},\m{W}^{l,2} \in \bC^{n\times n}$,
and $\m{\xi} \in \bR^{2n-1}$ has zero entries except $\xi_n = 1$.
By comparing \eqref{eq:dualSDP} and \eqref{eq:p1_dual}, we conjecture that the constraint $\norm{\m{V}}^*_{\cA}\le 1$ can be cast (up to sum-of-squares relaxations) as the four constraints in \eqref{eq:dp2} and \eqref{eq:dp4}.

The problem in \eqref{eq:dualSDP} is a SDP. Once \eqref{eq:dualSDP} is solved and the dual solution $\m{V}^*$ is obtained, we can evaluate the vector dual polynomial $Q\sbra{f} = \m{a}\sbra{f}^H\m{V}^* $ and localize the frequencies that satisfy $\norm{Q\sbra{f}}_1=1$ according to \cref{prop:ULA}. This provides another approach to frequency retrieval besides the Vandermonde decomposition method given the primal solution.

\subsection{ADMM Algorithm and Computational Complexity} \label{sec:ADMM}
The existing off-the-shelf SDP solvers, e.g., SDPT3 \cite{toh2012implementation}, are based on the interior point method (IPM), which has high computational complexity. We present a reasonably fast algorithm based on ADMM \cite{boyd2010distributed}. Taking the SLRA problem in \eqref{eq:p_SDP_completion} for example. To apply the ADMM, we introduce multiple auxiliary Hermitian matrix variables $\lbra{\m{Q}^l}_{l=1}^L$ and write \eqref{eq:p_SDP_completion} as:
\equ{ 
	\begin{split}
		& \min_{\m{t},\m{Z},\lbra{\m{Q}^l \succeq \m{0}}_{l=1}^{L}} \frac{1}{2}\frobn{\m{Y}_{\Omega}-\m{Z}_{\Omega}}^2 +  \frac{\tau}{n} \tr\sbra{\cT\m{t}}, \\
		& \st \m{Q}^l = \begin{bmatrix} \cT\overline{\m{t}} & \cH\overline{\m{Z}}_{:,l} \\ \cH\m{Z}_{:,l} & \cT\m{t} \end{bmatrix}, \; l = 1,\ldots, L. 
	\end{split}
	\label{eq:sdp2_convex}
}
The augmented Lagrangian function is then given by:
\equ{
	\begin{split}
		& \cL \sbra{\m{Z},\m{t},\lbra{\m{Q}^l},\lbra{\m{\Lambda}^l} } \\
		& =  \frac{1}{2}\frobn{\m{Y}_{\Omega}-\m{Z}_{ \Omega}}^2 + \tau t_n  + \sum^L_{l=1} \inp{ \m{\Lambda}^l,\m{Q}^l - \begin{bmatrix} \cT\overline{\m{t}} & \cH\overline{\m{Z}}_{:,l} \\ \cH\m{Z}_{:,l} & \cT\m{t} \end{bmatrix} }_{\bR} \\
		& \quad + \frac{\rho}{2} \sum^L_{l=1} \frobn{\m{Q}^l - \begin{bmatrix} \cT\overline{\m{t}} & \cH\overline{\m{Z}}_{:,l} \\ \cH\m{Z}_{:,l} & \cT\m{t} \end{bmatrix} }^2, \\ \nonumber
		& =  \frac{1}{2}\frobn{\m{Y}_{\Omega}-\m{Z}_{ \Omega}}^2 + \tau t_n + \frac{\rho}{2} \sum^L_{l=1} \frobn{\m{Q}^l - \begin{bmatrix} \cT\overline{\m{t}} & \cH\overline{\m{Z}}_{:,l} \\ \cH\m{Z}_{:,l} & \cT\m{t} \end{bmatrix} + \rho^{-1} \m{\Lambda}^l }^2 - \frac{1}{2\rho} \sum^L_{l=1} \frobn{ \m{\Lambda}^l}^2, \nonumber
	\end{split}
}
where $\rho>0$ is a penalty parameter and $\lbra{\m{\Lambda}^l}^L_{l=1}$ is multiple Hermitian Lagrangian multipliers. 

Assume that at iteration $m$ we have computed $\m{Z}^m$, $\m{t}^m$, and $\m{\Lambda}^{m,l}$ for $l=1,\ldots,L$, the $(m+1)$-th iteration of ADMM is given by:
{
	\lentwo\equa{
		\lbra{\m{Q}^{m+1,l}}  &=&  \argmin_{\m{Q}^l \succeq \m{0}} \cL \sbra{\m{Z}^m,\m{t}^m,\lbra{\m{Q}^l},\lbra{\m{\Lambda}^{m,l}} }, \label{eq:Z} \\
		\sbra{\m{Z}^{m+1},\m{t}^{m+1}}  &=&  \argmin_{\m{Z},\m{t}} \cL \sbra{\m{Z},\m{t},\lbra{\m{Q}^{m+1,l}},\lbra{\m{\Lambda}^{m,l}} }, \label{eq:Xt} \\
		\m{\Lambda}^{m+1,l} &=& \m{\Lambda}^{m,l} + \rho \sbra{\m{Q}^{m+1,l} - \begin{bmatrix} \cT\overline{\m{t}^{m+1}} & \cH\overline{\m{Z}^{m+1}_{:,l}} \\ \cH\m{Z}^{m+1}_{:,l} & \cT\m{t}^{m+1} \end{bmatrix} }, \, l=1,\ldots,L. \label{eq:Lambda}
	}
}

For the first subproblem in \eqref{eq:Z}, we have the updates:
\equ{
	\m{Q}^{m+1,l} = \cP\sbra{ \begin{bmatrix} \cT\overline{\m{t}^{m}} & \cH\overline{\m{Z}^{m}_{:,l}} \\ \cH\m{Z}^{m}_{:,l} & \cT\m{t}^{m} \end{bmatrix} - {\rho^{-1}}\m{\Lambda}^{m,l}}, \quad l = 1,\ldots,L, \label{eq:Z_HT}
}
where $\cP$ denotes the orthogonal projection of a Hermitian matrix onto the PSD cone by forming the eigen-decomposition and setting all but the positive eigenvalues to zero \cite{dax2014low}.

For the second subproblem in \eqref{eq:Xt}, the variables $\m{Z}$ and $\m{t}$ can be separately solved for in closed form. Denote a Hermitian matrix $\m{P}^l=\m{Q}^{m+1,l}+\rho^{-1}\m{\Lambda}^{m,l}$ and write $\m{P}^l = \begin{bmatrix} \m{P}^{l,1} & \sbra{\m{P}^{l,3}}^H \\ \m{P}^{l,3} & \m{P}^{l,2} \end{bmatrix}$ as a block matrix like the Hankel-Toeplitz matrix. For $1\le j \le N$, we have
\equ{
	z_{j,l}^{m+1}  = \sbra{\omega_j + 2\rho d_j}^{-1} \mbra{y_{j,l} + 2\rho \sbra{\cH^H\m{P}^{l,3}}_j  }, \label{eq:x_HT}
}
and for $N < j \le 2n-1 $,
\equ{
	z_{j,l}^{m+1}  = d_j^{-1} \sbra{ \cH^H\m{P}^{l,3} }_j, \label{eq:x_HT2}
}
for each $l=1,\ldots,L$, where $\m{d} = \mbra{1,2,\dots,n,n-1,\dots,1}$ and $\omega_j = 1$ if $j\in \Omega$ or $0$ otherwise.
The update for $\m{t}$ is given by:
\equ{
	\m{t}^{m+1}  = \frac{1}{2L} \sbra{\diag\sbra{\m{d}}}^{-1} \sbra{ \cT^H \lbra{\sum^L_{l=1} \sbra{ \overline{\m{P}^{l,1}} + \m{P}^{l,2}} } - \frac{\tau}{\rho} \m{\xi}  }. \label{eq:t_HT}
}
According to \cite{boyd2010distributed}, the ADMM algorithm converges to the optimal solution of the convex optimization problem in \eqref{eq:p_SDP_completion}. The ADMM implementation of \eqref{eq:p_SDP} is similar to that of \eqref{eq:p_SDP_completion} and hence is omitted.

\begin{algorithm}
	\caption{SLRA-based ANM for CAFE (SACA) using the ADMM algorithm}
	\label{alg:SACA}
	\begin{algorithmic}
		\STATE{Initialize $\m{Z}^1_{\Omega} = \m{Y}_{\Omega}$, $\m{t}^1 = \m{0}$, and $\lbra{\m{\Lambda}^{1,l} = \m{0}}$}
		\WHILE{not converged}
		\STATE{Update $\lbra{\m{Q}^l}$ using \eqref{eq:Z_HT}}
		\STATE{Update $\sbra{\m{Z},\m{t}}$ using \eqref{eq:x_HT}, \eqref{eq:x_HT2}, and \eqref{eq:t_HT}}
		\STATE{Update $\lbra{\m{\Lambda}^l}$ using \eqref{eq:Lambda}}
		\ENDWHILE
		\STATE{Compute the Vandermonde decomposition of $\cT\m{t}$ in \eqref{eq:T2} using root-MUSIC \cite{barabell1983improving}}
		\RETURN Solution $(\m{f}^*,\m{b}^*)$ as estimate of $(\m{f},\m{b})$
	\end{algorithmic}
\end{algorithm}

We summarize the proposed SACA approach using the ADMM algorithm in \Cref{alg:SACA} and analyze its complexity. Taking the SLRA problem in \eqref{eq:p_SDP_completion} for example. It has $d = \mathcal{O}\sbra{N}$ free variables and $L$ LMIs, and the $l$-th LMI has size of $k_l\times k_l$ with $k_l = \mathcal{O}\sbra{N}$. It follows from \cite{ben2001lectures} that the IPM for \eqref{eq:p_SDP_completion} has computational complexity on the order of
$
\sbra{1+\sum^L_{l=1}k_l}^{\frac{1}{2}}  d \sbra{  d^2 + d \sum^L_{l=1}k^2_l + \sum^L_{l=1}k_l^3  } = \mathcal{O} \sbra{L^{1.5}N^{4.5}}.
$
In contrast to this, the ADMM algorithm has a per-iteration complexity of $\mathcal{O} \sbra{L N^3}$ that is dominated by the eigen-decompositions for updating $\lbra{\m{Q}^l}$ in \eqref{eq:Z_HT}. This complexity can be further reduced to $\mathcal{O} \sbra{N^3 + L N^2}$ by updating $\lbra{\m{Q}^l}$ in parallel where $LN^2$ arises from the update of $\m{t}$ in \eqref{eq:t_HT}.

\section{Theoretical Guarantees} \label{sec:theory}
\subsection{Dual Certificate}
The following proposition provides a dual certificate for validating optimality of a solution to the problem in \eqref{eq:p1}. Its proof is similar to \cite[Prop II.4]{tang2013compressed} \cite{yang2016exact} and will be omitted.
\begin{proposition} \label{prop:ULA}
	In the full data case, $\m{X}^\star =\sum_{k=1}^K b_k \m{a}(f_k) e^{i\m{\phi}_k}$ with $\m{\phi}_k = \mbra{\phi_{k,1},\ldots,\phi_{k,L}}$ is the unique atomic decomposition satisfying that $\norm{\m{X}^\star}_{\cA} = \sum_{k=1}^K b_k$ if there exists a vector-valued dual polynomial $Q:\bT\to \bC^{1\times L}$,
	\equ{
		Q\sbra{f} = \m{a}\sbra{f}^H\m{V} \label{eq:Q_fulldata}
	} 
	satisfying that
	{
		\lentwo\equa{
			Q\sbra{f_k} &=& \frac{1}{L}e^{i\m{\phi}_k}, \quad f_k \in \Upsilon, \label{eq:condtion1} \\
			\norm{Q\sbra{f}}_1 &<& 1, \quad f \in \bT \setminus \Upsilon, \label{eq:condtion2} 
		}
	}where $\m{V}$ is an $N\times L$ matrix and $ \Upsilon \subset \bT$ denotes the frequency set of $\m{X}^\star$. In the missing data case, $\m{X}^\star=\sum_{k=1}^K b_k \m{a}(f_k) e^{i\m{\phi}_k}$ is the unique optimizer of \eqref{eq:p1} if $\lbra{\m{a}_{\Omega}\sbra{f_k}}_{f_k\in \Upsilon} $ are linearly independent, where $\m{a}_{\Omega}\sbra{\cdot}$ is a subvector of $\m{a}\sbra{\cdot}$ indexed by $\Omega$, and there exists $Q\sbra{f}$ in \eqref{eq:Q_fulldata} satisfying \eqref{eq:condtion1}, \eqref{eq:condtion2} and the additional constraint that $\m{V}_j = \m{0}, \; j\not \in \Omega. \label{eq:V_missing}$
\end{proposition}

\subsection{Advantage of Using the CA Property}
We study the advantage of using the CA property by relating SACA to the previous ANM in \cite{yang2016exact,yang2018sample,li2015off,fernandez2016super,steffens2018compact,li2018atomic} that does not consider the CA property. Take the full data case as an example. If we apply the previous ANM to the CAFE problem concerned in the present paper, according to \cite{yang2016exact,li2015off}, the frequencies can be exactly recovered if there exists a dual certificate $\breve{Q}\sbra{f} = \m{a}\sbra{f}^H\m{V}$ satisfying that
{
	\lentwo\equa{
		\breve{Q}\sbra{f_k} &=& \frac{1}{\sqrt{L}} e^{i\m{\phi}_k} , \quad f_k \in \Upsilon, \label{eq:Qb1} \\
		\norm{\breve{Q}\sbra{f}}_2 &<& 1, \quad f \in \bT \setminus \Upsilon. \label{eq:Qb2}
	}
} 
Then, we have the following result.
\begin{lemma} \label{lem:weaker}
	Suppose that $\breve{Q}\sbra{f}$ is a dual certificate for ANM that satisfies \eqref{eq:Qb1} and \eqref{eq:Qb2}. Then, $Q\sbra{f} = \breve{Q}\sbra{f}/\sqrt{L}$ is a dual certificate for SACA that satisfies \eqref{eq:condtion1} and \eqref{eq:condtion2}.
\end{lemma}
\begin{proof}
	Since $Q\sbra{f} = \breve{Q}\sbra{f}/\sqrt{L}$, \eqref{eq:condtion1} is a direct consequence of \eqref{eq:Qb1}. Moreover, the inequality in \eqref{eq:condtion2} holds since
	$
	\norm{Q\sbra{f}}_1 \le \sqrt{L} \twon{Q\sbra{f}} = \norm{\breve{Q}\sbra{f}}_2 < 1,
	$
	completing the proof.
\end{proof}

We show by \cref{lem:weaker} that the conditions required for the dual certificate of SACA are weaker than those for ANM without using the CA property, implying the advantage of using the CA property. 

\subsection{Exact Recovery in the Noiseless Case} \label{sec:noiseless}
Based on \cref{lem:weaker}, we can translate theoretical results for ANM to SACA. The following theorem considers the full data case and is a result of combining \cref{lem:weaker} and \cite[Theorem 4]{yang2016exact}.

\begin{theorem} \label{thm:ULA}
	If $N\ge 257$ and the frequency support $\Upsilon = \lbra{f_k} \subset \bT$ satisfies the minimum separation condition $\Delta_{f} \triangleq \min_{p\neq q} \abs{ f_{p} - f_{q} } > \frac{1}{\left \lfloor (N-1)/4 \right \rfloor}, \label{eq:dist}$
	where the distance is wrapped around on the unit circle. Then $\m{X}^\star = \sum_{k=1}^K b_k \m{a}(f_k) e^{i\m{\phi}_k}$ is the unique atomic decomposition satisfying that $\norm{\m{X}^\star}_{\cA} = \sum_{k=1}^K b_k$.
\end{theorem} 

In the missing data case, we expect that theoretical guarantees can be derived for SACA by combining \cref{lem:weaker} and existing results for ANM (see, e.g., \cite[Theorem 1]{yang2018sample}). Differently from \cite[Theorem 4]{yang2016exact} in which no assumptions are made for the phases, the phases in \cite[Theorem 1]{yang2018sample} are assumed to lie uniformly on the unit hypersphere, which however cannot be satisfied for the CA signals concerned in the present paper. To resolve this problem, we show the following lemma that is in parallel with \cite[Lemma 4]{yang2018sample}, to be specific, the two lemmas contain the same conclusion but different assumptions on the phase matrix $\m{\Phi}$.

\begin{lemma} \label{lem:1}
	Let $\m{0}\neq \m{w}\in \bC^K$ and $\m{\Phi}\in \bC^{K\times L}$ with $\Phi_{k,l}=e^{i\phi_{k,l}}$. Assume that $\lbra{\phi_{k,l}}$ are sampled i.i.d. from the uniform distribution on $\mbra{0,2\pi}$. Then, for all $u\ge \norm{\m{w}}_2 $ and a constant $c$, we have
	$
	\bP\lbra{ \twon{\frac{1}{\sqrt{L}} \m{\Phi}^H\m{w}} \ge u } \le  e^{-c L\sbra{\frac{u}{\twon{\m{w}}}-1}^2}. $
\end{lemma}
\begin{proof}
	See \cref{append:thm}. 
\end{proof}

Making use of \cref{lem:1}, instead of \cite[Lemma 4]{yang2018sample}, and using the proof techniques of \cite[Theorem 1]{yang2018sample}, we can show the following theorem, of which the detailed proof will be omitted.
\begin{theorem} \label{thm:SLA}
	Given the noiseless observation $\m{X}^\star_{\Omega}$ in the missing data case and the phase matrix $\m{\Phi}$ with $\Phi_{k,l}=e^{i\phi_{k,l}}$. Assume that $\lbra{\phi_{k,l}}$ are sampled i.i.d. from the uniform distribution on $\mbra{0,2\pi}$ and the frequencies $\lbra{f_k}$ satisfy the minimum separation condition. 
	Then, with probability at least $1-\delta$, there exists a numerical constant $C$ such that 
	$
	M \ge C \max \lbra{ \log^2 \frac{N}{\delta}, K \sbra{\log \frac{K}{\delta} }  \sbra{1+\frac{1}{L}\log \frac{N}{\delta}} } \label{eq:M_bound}
	$
	is sufficient to guarantee that $\m{X}^\star = \sum_{k=1}^K b_k \m{a}(f_k) e^{i\m{\phi}_k}$ is the unique optimizer to \eqref{eq:p1}.
\end{theorem}

\cref{thm:SLA} shows that a resolution inversely proportional to the signal length $N$ and a sample size $M$ scaling with the model order $K$ are sufficient to guarantee exact frequency estimation. 
While the shown resolution and sample complexity for CA signals coincide with those for general multichannel signals, we expect that SACA has better empirical performance than the previous ANM due to the benefit of using the CA property shown by \cref{prop:ULA} and \cref{lem:weaker}. But we also note that the sample complexity in \cref{thm:SLA} cannot be substantially improved in future studies. In particular, note that there are $K(L+2)$ real unknowns including $\lbra{f_k}, \lbra{b_k}$ and $\lbra{\phi_{k,l}}$ in the signal model in \eqref{eq:signal}. The minimum required number of complex-valued samples is thus $K(L+2)/2$, and the number of samples per channel for any method must satisfy $M \ge \frac{1}{2L}K\sbra{L+2} = K \sbra{\frac{1}{2}+\frac{1}{L}}$. The second term of the bound of $M$ in \cref{thm:SLA} is merely logarithmic factors greater than this information theoretic rate.

\subsection{Choosing the Regularization Parameter in the Noisy Case} \label{sec:tau}
Define the set
\equ{
	\cA_{\Omega} \triangleq \lbra{ \m{a}_{\Omega}(f) \m{\psi}: f \in \bT, \m{\psi} \in \bC^{1 \times  L}, \abs{\psi_l} = 1, l = 1,\ldots,L},
}
and the atomic norm $\norm{\cdot}_{\mathcal{A}_{\Omega}}$ associated with it, as in \eqref{eq:atomic1}. Suppose that $\Omega$ is sorted ascendingly and denote the range of the sampling period $\overline{N} = \Omega_M - \Omega_1 + 1 \le N$. It follows from \cite[Theorem 1]{bhaskar2013atomic} that the estimate $\widehat{\m{X}}$ given by the solution in \eqref{eq:atomic_min2} with $\tau \ge \bE \norm{\m{E}_{\Omega}}^*_{\cA_\Omega}$ has the expected (per-element)  mean squared error $\frac{1}{ML} \bE \frobn{\widehat{\m{X}}_{\Omega} - \m{X}^\star_{\Omega}}^2 \le \frac{\tau}{ML} \norm{\m{X}^\star}_{\cA}$, where $\norm{\cdot}_{\mathcal{A}_{\Omega}}^*$ is the dual norm of $\norm{\cdot}_{\mathcal{A}_{\Omega}}$, as in \eqref{eq:dualnorm}. 
To determine $\tau$, we have the following result.
\begin{theorem} \label{thm:noisy}
	If the entries of $\m{E}$ obey i.i.d. Gaussian distribution $\cC\cN(0,\sigma)$, then the expected dual norm is bounded as:
	\equ{
		\bE \norm{\m{E}_{\Omega}}^*_{\cA_\Omega} \le C \sqrt{\sigma M L (L+\ln \overline{N})},  \label{eq:dualE_SLA}
	}
	where the constant $C=p_1p_2 / \sbra{p_1 p_2 - p_1 - p_2}$ with $p_1 = 4\ln \sbra{L \ln 8\pi + \ln \overline{N}}$ and $p_2 = 4$.
\end{theorem}
\begin{proof}
	See \cref{append:noisy}.
\end{proof}

It follows from \cref{thm:noisy} that the regularization parameter $\tau$ can be given as the upper bound in \eqref{eq:dualE_SLA}. In this case, the problem in \eqref{eq:atomic_min2} produces a consistent estimate (on $\m{\Omega}$) when $K=o\sbra{\sqrt{ML/(L+\ln \overline{N})}}$. In the single-channel case of $L=1$, the bound in the full data case is consistent with that in \cite{bhaskar2013atomic}. 

\section{Numerical Results} \label{sec:Sim}
In this section, we present numerical results to evaluate the performance of SACA and validate our theoretic analyses. 
All algorithms are implemented using MATLAB (R2021b) on a 64-bit Windows server with an Intel Xeon Gold 6133 CPU at 2.5 GHz and 192 GB of RAM.
In SACA, the SLRA problems in \eqref{eq:p_SDP} and \eqref{eq:p_SDP_completion} are solved using an IPM with SDPT3 solver (version 4.0) \cite{toh2012implementation} in the CVX toolbox \cite{cvx}, using default parameters, e.g., the Helmberg-Kojima-Monteiro (HKM) direction.
We set the dimension parameter $n= \lceil (N+1)/2 \rceil$ in the following experiments by default unless otherwise specified.

\subsection{The Noiseless Case}
In \emph{Experiment 1}, we study the number of identifiable frequencies via comparing SACA with ANM \footnote{\url{https://sites.google.com/site/zaiyang0248/publications}} \cite{yang2016exact} (without using the CA property), ACMA \footnote{\url{https://sps.ewi.tudelft.nl/Repository/repitem.php?id=15&ti=3}} \cite{leshem1999direction,van1996analytical}, and SMART \cite{wu2022Direction}. Since the model order $K$ is unknown for SACA, we propose a checking mechanism for frequency parameter identifiability from the noiseless observation $\m{X}^\star$ according to \cref{thm:SLRA}. In particular, we first solve the SLRA problem of $\m{X}^\star$ in \eqref{eq:SDP}. If $\rank\sbra{\cT\m{t}^*}=n$ for the numerical solution, then we increase $n$; otherwise the algorithm is terminated and the frequencies are retrieved from $\cT\m{t}^*$ via root-MUSIC. The maximum value of $n$ is set to $2N-1$, while it is theoretically shown in \cite{wax1992unique,williams1992resolving,valaee1994alternative} that at most a number $2N-3$ of frequencies can be uniquely identified from a length-$N$ CA signal when the full data is available. 

In our experiment, we consider the full data case with $N=5$ and two settings of the number of channels, $L=10$ and $L=100$. A number $K$ of frequencies with $K \in \lbra{2,\ldots,8}$ are given by $\lbra{-0.45+(k-1)/K, k=1,\ldots,K}$. We say that the frequencies are successfully estimated by an approach if the root mean squared error (RMSE) $\frac{1}{\sqrt{K}}\twon{\m{f}-\widehat{\m{f}}} \le 10^{-4}$, where $\widehat{\m{f}}$ is the vector of estimated frequencies. For each $K$, we calculate the success rate averaged over $100$ Monte Carlo trials.
The true $K$ is fed into ACMA and SMART.
Our simulation results are presented in \cref{fig:K_N}. It is seen that for SACA, SMART, and ACMA, the number of identifiable frequencies increases as $L$ increases from $10$ to $100$. As expected, ANM can identify at most $N-1=4$ frequencies since it does not use the CA property. The number of frequencies identified by the two-step method ACMA is consistent with the assumption $K \le \min(N, \sqrt{L})$ made in \cite{van1996analytical} because it cannot use the Vandermonde structure in its first step (note that other ACMA-based variants suffer from similar limitations). In contrast to them, both SACA and SMART can identify $K>N=5$ frequencies. This partly verifies the conclusion for frequency parameter identifiability of CA signals in \cite{wax1992unique,williams1992resolving,valaee1994alternative}. It is also seen that SACA outperforms SMART. This is because SMART suffers from convergence issues when $K$ is large. In contrast to SMART, SACA is a convex approach and can always find the global optimizer. 

\begin{figure}[htb]  
	\centering
	\subfigure[$L=10$] {\includegraphics[width=6cm]{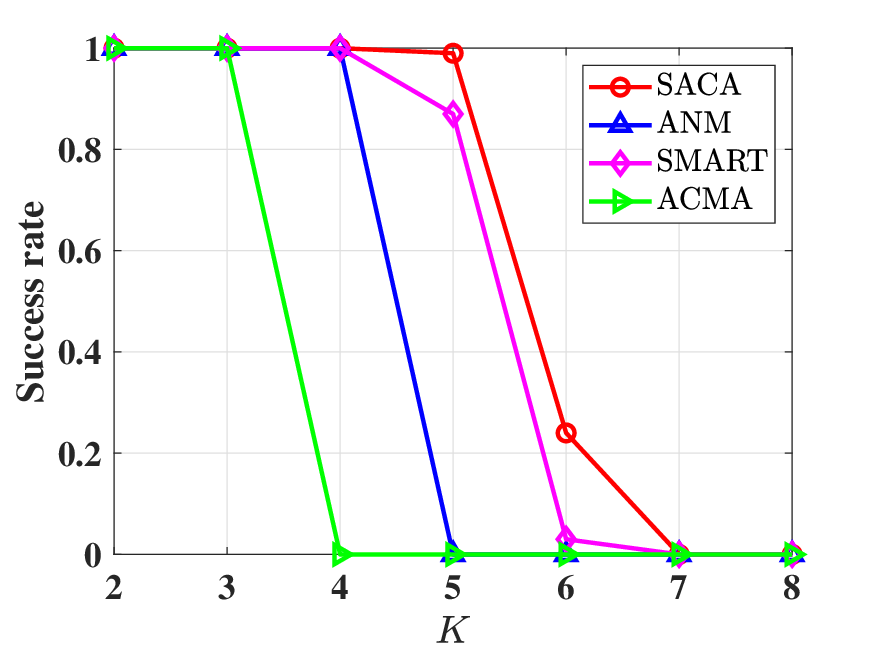}}
	\subfigure[$L=100$]
	{\includegraphics[width=6cm]{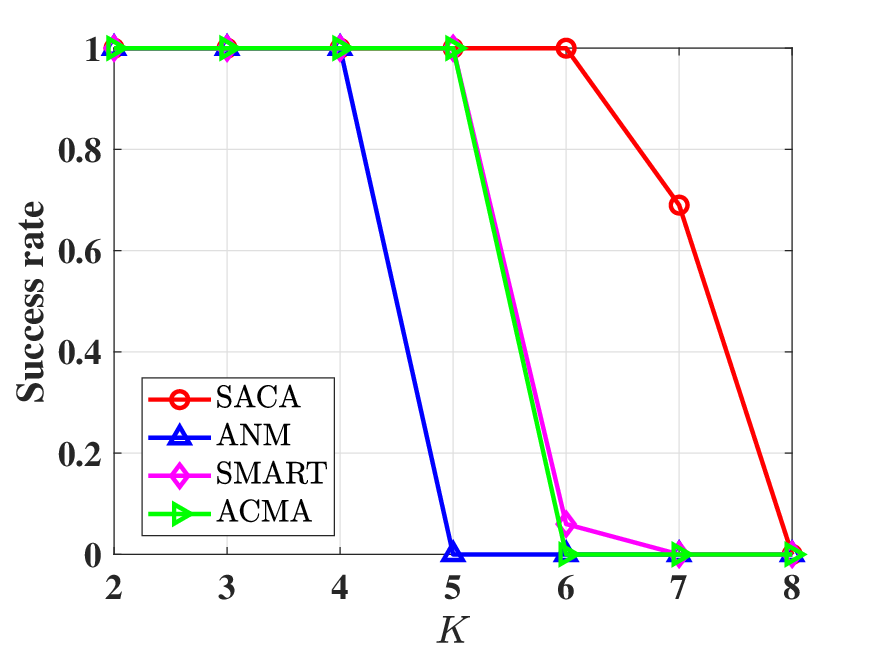}}
	\caption{Success rate of frequency estimation versus $K$ in the full data case with $N=5$.}
	\label{fig:K_N}
\end{figure}

\begin{figure*}[htb] 
	\centering
	\subfigure[] {\includegraphics[width=4cm]{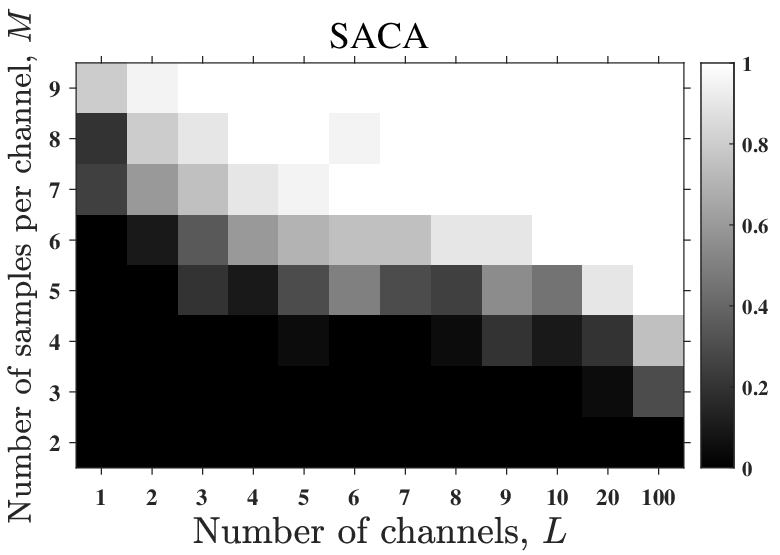}}
	\subfigure[] {\includegraphics[width=4cm]{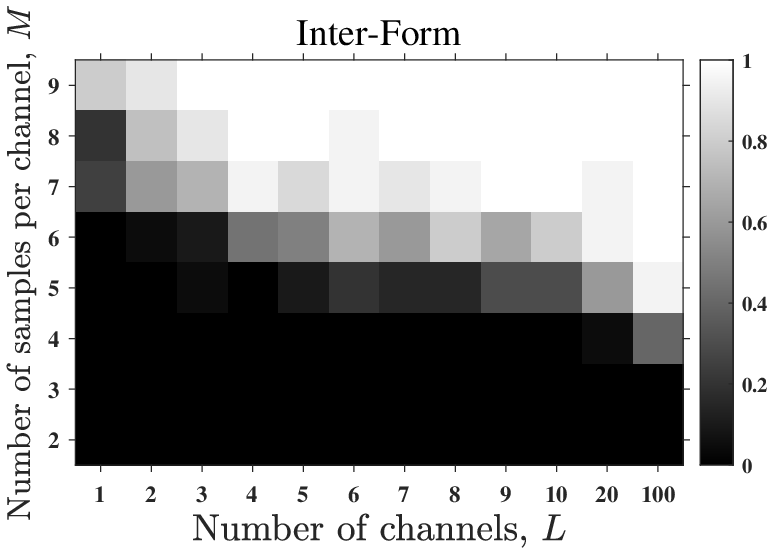}}
	\subfigure[] {\includegraphics[width=4cm]{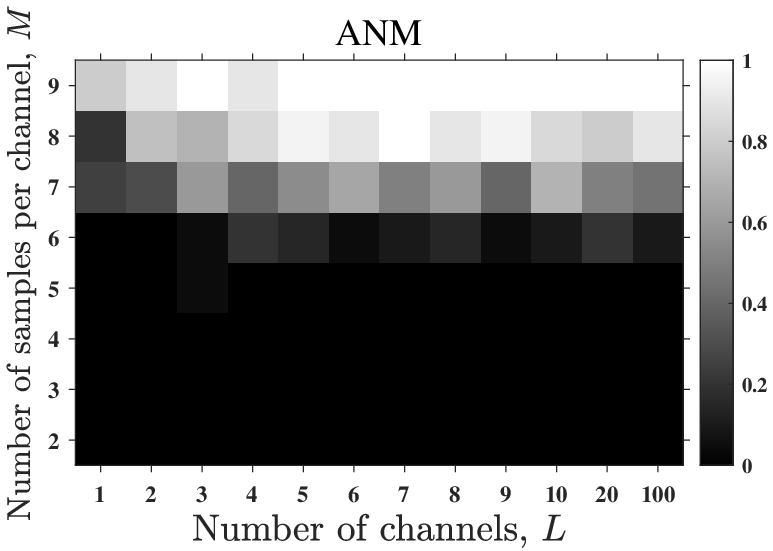}}
	
	\subfigure[] {\includegraphics[width=4cm]{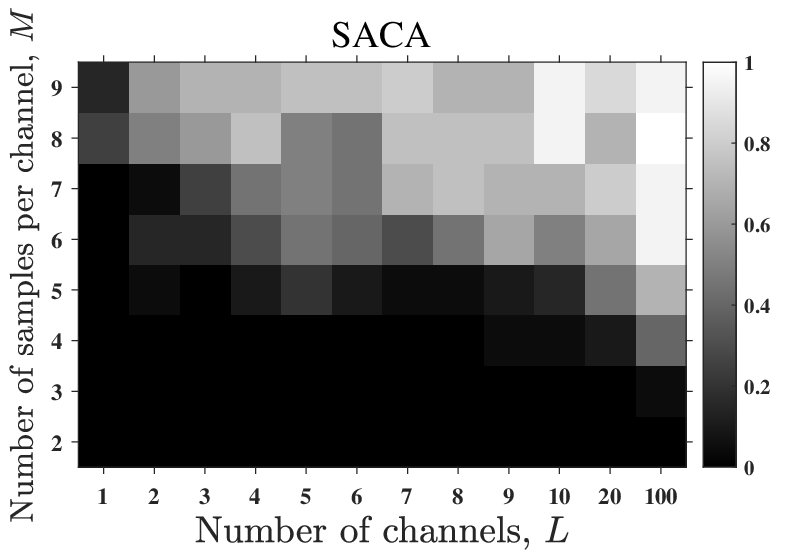}}
	\subfigure[] {\includegraphics[width=4cm]{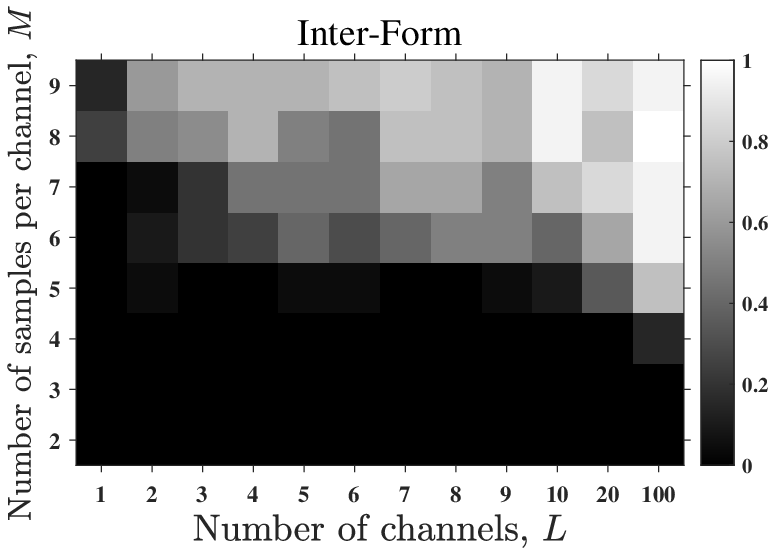}}
	\subfigure[] {\includegraphics[width=4cm]{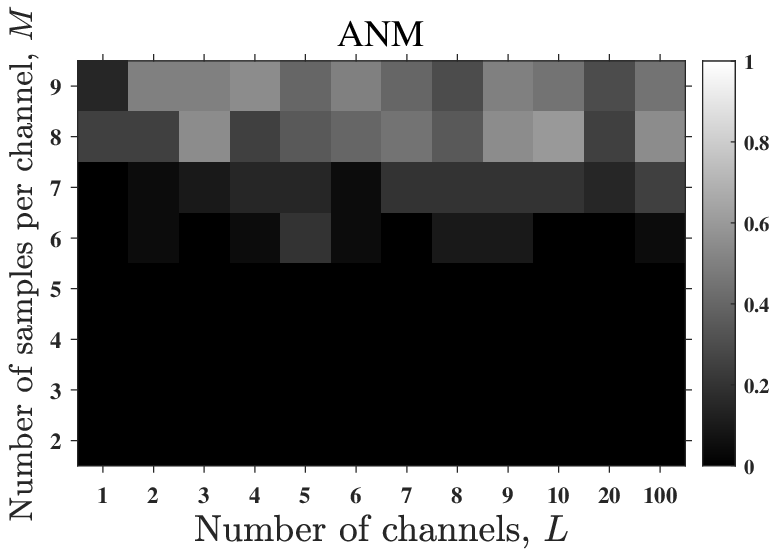}}
	\caption{Phase transition results in the missing data case when $N=11$ and $K=4$. (a) SACA with $\Delta_{f} = 1.2/N$. (b) Inter-Form with $\Delta_{f} = 1.2/N$.  (c) ANM with $\Delta_{f} = 1.2/N$. (d) SACA without separation. (e) Inter-Form without separation. (f) ANM without separation. White means complete success and black means complete failure.}
	\label{fig:PhaseFig}
\end{figure*}

In \emph{Experiment 2}, we consider the missing data case and study the success rates of the proposed SACA in \eqref{eq:p1} as compared to ANM \cite{yang2016exact} and the intermediate formulation in \eqref{eq:SDP_T_var}, designated as Inter-Form.  
In particular, $M$ entries of the observation set $\Omega$ are randomly selected from $\lbra{1,\ldots,N}$ with $N=11$. A number $K=4$ of frequencies are randomly generated with or without a minimum separation $\Delta_{f} = 1.2/N$. Amplitudes $\lbra{b_k}$ and phases $\lbra{\phi_{k,l}}$ are also randomly generated. We consider $M\in \lbra{2,\ldots,9}$ and $L\in \lbra{1,\ldots,10,20,100}$. The success rate is calculated by averaging over $20$ Monte Carlo trials for each combination $\sbra{M,L}$. Our results are presented in \cref{fig:PhaseFig} where phase transition behaviors are observed. It is seen that SACA has a much larger success phase than ANM, which verifies the advantage of using the CA property shown in \cref{lem:weaker}.
The required samples per channel for exact frequency estimation using SACA decreases with the number of channels, verifying \cref{thm:SLA}. When the minimal separation $\Delta_{f}$ is absent, both the performances of SACA and ANM degrade, while SACA still outperforms ANM. The performance of Inter-Form is better than ANM but inferior to SACA for both well-separated and random frequencies.

To understand the performance of Inter-Form presented in the last experiment, in \emph{Experiment 3}, we study the capability of SACA, ANM \cite{yang2016exact}, and Inter-Form in exploiting the CA structure. Specifically, we consider $N=11$, $M=6$, and $L=10$. The frequencies are generated randomly with a minimum separation $1.2/N$. A number $K=4$ of CA signals are generated as in \emph{Experiment 2}. Our results are presented in \cref{fig:threeSDP_CM_nonCM}. It is seen that both of the recovered amplitudes of ANM and Inter-Form are not constant, while the fluctuation of that by Inter-Form is smaller than that by ANM. In contrast to this, the recovered amplitudes of our SACA are constant and consistent with the ground truth. We therefore conclude that Inter-Form can indeed exploit the CA structure to some extent, but differently from the SLRA problem in \cref{thm:SLRA}, it is not an exact characterization of the CA atomic norm.

\begin{figure*}[htb]
	\centering
	\includegraphics[width=9cm]{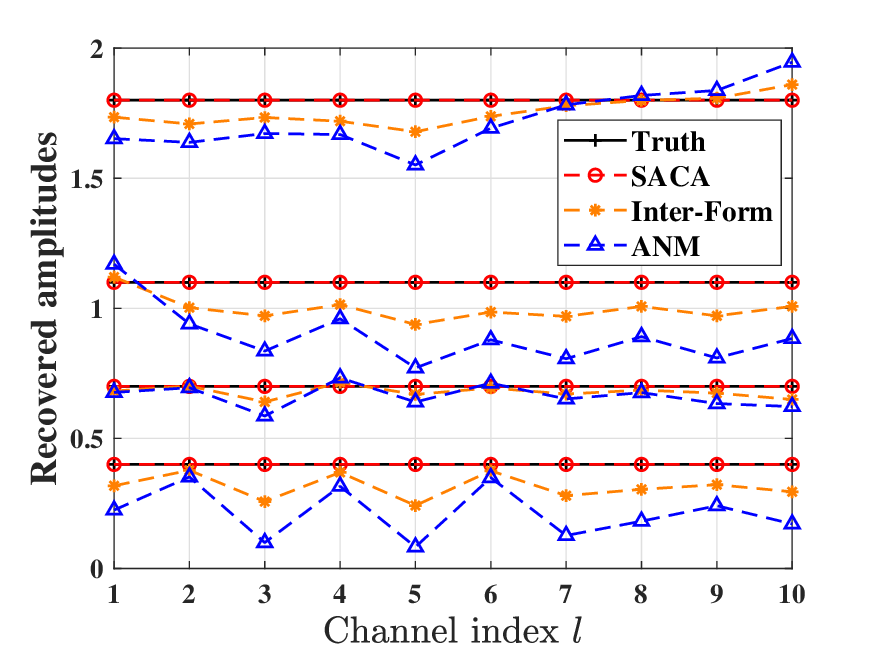}
	\caption{The recovered amplitudes versus the channel index $l$.} \label{fig:threeSDP_CM_nonCM}
\end{figure*} 

\subsection{The Noisy Case}
In this subsection, the methods that we use for comparison include two representative two-step methods, ACMA \cite{van1996analytical,leshem1999direction} and ZF-ACMA \cite{van2001asymptotic}, and two MLE methods, Newton’s method (NM) \cite{leshem2000maximum} and SMART \cite{wu2022Direction}. Besides, the ANM method using the ADMM algorithm \footnote{\url{http://users.ece.cmu.edu/~yuejiec/publications.html}} for general multichannel signals \cite{li2015off} is also considered.
The CRBs for CA and general multichannel signals \cite{leshem1999direction}, \cite{stoica1989music} are presented as benchmarks.
The signal-to-noise ratio (SNR) is defined as $10\log_{10}\sbra{\frobn{\m{X}^\star}^2/\frobn{\m{E}}^2}$.
For ACMA and ZF-ACMA, one-dimensional grid search with a uniform grid of size $10^4$ is used to estimate the frequencies from the estimated matrix as suggested in \cite{leshem1999direction}. For NM, we initialize it by ACMA as in \cite{leshem2000maximum}.
Since ACMA, ZF-ACMA, NM, and SMART need the model order, we feed the true $K$ to them. For SACA and ANM, the noise variance $\sigma$ is used to calculate the regularization parameters. Since the model order might not be correctly determined, to make a fair comparison, we compute the $K$ frequency estimates using root-MUSIC for SACA and ANM. A total of $200$ Monte Carlo trials are conducted and then averaged to produce each simulated point in the following figures.

In \emph{Experiment 4}, we test the effect of the frequency separation $\Delta_f$ on the frequency estimation performance of SACA. For the full data case with $N=11$, we consider $K = 3$ frequencies given by $\lbra{-0.01,-0.01+\Delta_f, 0.35}$ and vary $\Delta_f$.
Complex white Gaussian noise is added to the observations with SNR $=20$dB. We set $L=20$.
Our simulation results are presented in \cref{fig:delta}.
It is seen that the CRB for CA signals is significantly lower than that for general signals when the frequency separation is small, implying great importance of using the CA property in this regime. 
ACMA and ZF-ACMA are sensitive to the frequency separation due to their suboptimal treatment of the Vandermonde and the grid search, to be specific, their accuracy fluctuates when the separation $\Delta_f \in [0.14, 0.2]/N$ and approaches the CRB for CA signals as the separation increases. NM performs similarly as ACMA. 
Thanks to the use of the CA property, the proposed SACA can outperform the CRB for general signals and always performs better than ANM. As compared to the two-step methods and NM, the convex approach SACA provides stable estimates and has higher resolution. SACA is inferior to the nonconvex approach SMART relying on the true model order.

\begin{figure*}[htb]
	\centering
	\includegraphics[width=9cm]{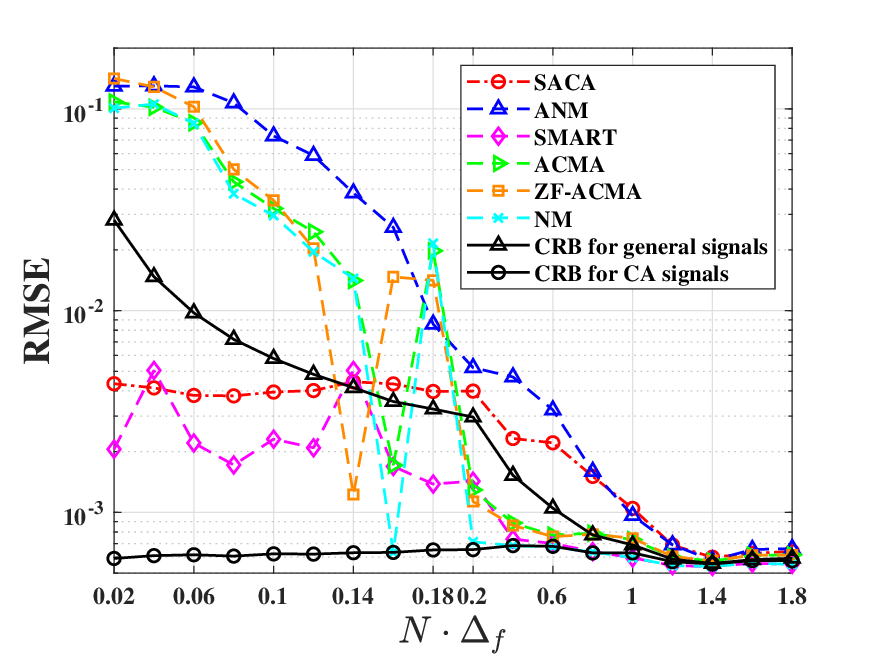}
	\caption{RMSE versus the frequency separation $\Delta_f$ when frequencies $\m{f} = \mbra{-0.01,-0.01+\Delta_f, 0.35}^T$.} \label{fig:delta}
\end{figure*} 

\begin{figure*}[htb] 
	\centering
	\subfigure[] {\includegraphics[width=6.1cm]{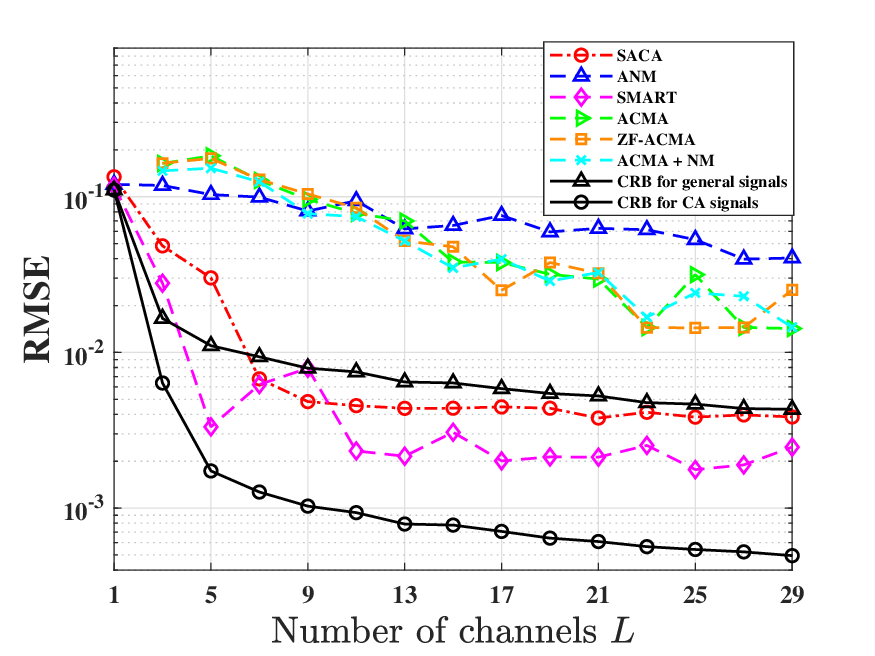}\label{fig:L_near} }
	\subfigure[] 
	{\includegraphics[width=6.1cm]{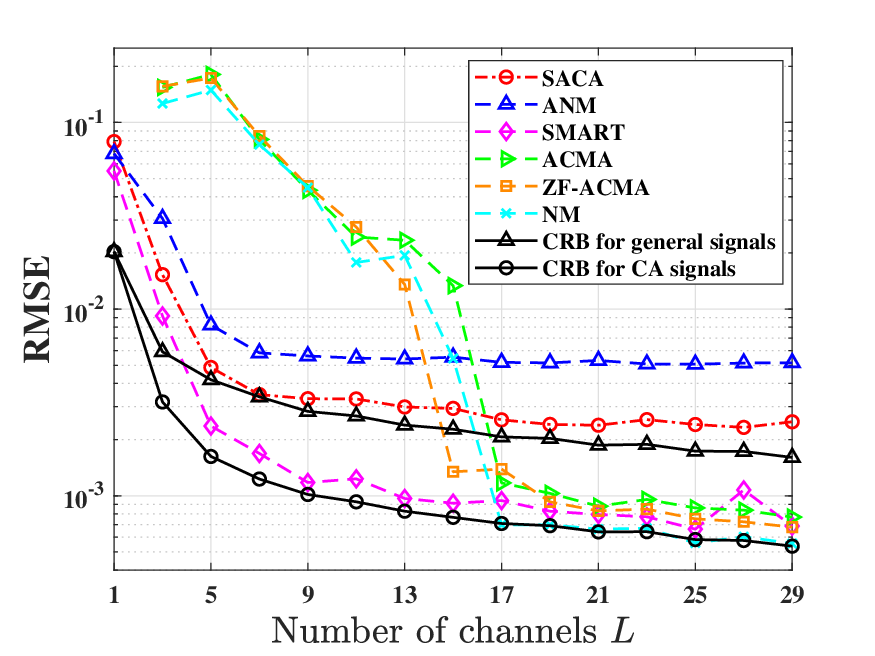} \label{fig:L_far} }
	\caption{RMSE versus the channels $L$ when (a) frequencies $\m{f}=\mbra{-0.01,0,0.35}^T$ and (b) frequencies $\m{f}=\mbra{-0.01,-0.01+0.3/N,0.35}^T$.}
\end{figure*}

In \emph{Experiment 5}, we study the RMSE performance versus the number of channels $L$. We consider $N=11$, $K=3$, SNR $=20$dB, and two sets of frequencies. 
Our simulation results regarding $\m{f} = \mbra{-0.01,0,0.35}^T$ and $\m{f} = \mbra{-0.01,-0.01+0.3/N,0.35}^T$ are presented in \cref{fig:L_near} and \cref{fig:L_far}, respectively. 
It is seen in \cref{fig:L_near} that the RMSEs of ACMA, ZF-ACMA, NM and ANM are always greater than the CRB for general signals under such closely located frequencies. The RMSE of SACA decreases with the increase of $L$ and is less than the CRB for general signals as $L\ge 7$. As the frequency separation increases in \cref{fig:L_far}, ACMA and ZF-ACMA get close to and NM achieves the CRB for CA signals when a sufficient number of channels are available. The proposed SACA always outperforms ANM, is superior to ACMA, ZF-ACMA and NM when the channels are limited, and is inferior to SMART. 

In \emph{Experiment 6}, we study the RMSE performance versus SNR. We consider the full data case with $N=11$ and the missing data case with $M = 9$ and $\Omega^c = \lbra{2,10}$. 
We set the frequencies $\m{f}=\mbra{-0.01,0,0.35}^T$ and channels $L=20$ and vary the SNR from $0$ to $30$ dB. Our simulation results regarding the two cases are presented in \cref{fig:SNR_ULA} and \cref{fig:SNR_SLA}, respectively. It is seen that the proposed SACA always performs better than ANM and outperforms the two-step methods and NM under low and moderate SNR.

\begin{figure*}[htb] 
	\centering
	\subfigure[]
	{\includegraphics[width=6.1cm]{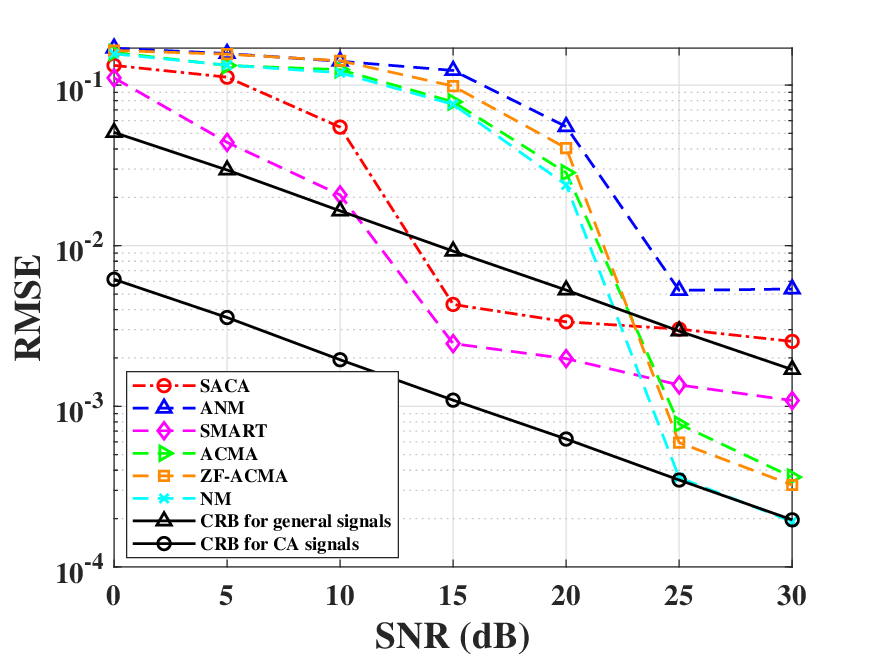} \label{fig:SNR_ULA} }
	\subfigure[] {\includegraphics[width=6.1cm]{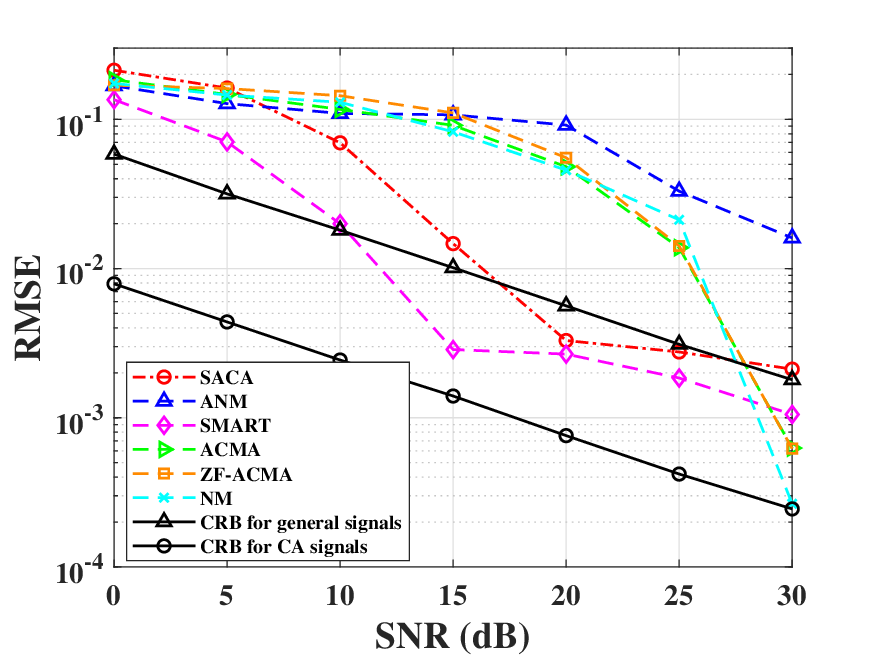} \label{fig:SNR_SLA} }
	\caption{RMSE versus the SNR (a) in the full data case and (b) in the missing data case with $M=9$.}
	\label{fig:sixFig}
\end{figure*}

In \emph{Experiment 7}, we study the average RMSE and CPU time for different $N$ over $100$ Monte Carlo trials. Specifically, we consider the full data case and set SNR $=20$dB, $L=30$, and $\m{f}=\mbra{-0.01,-0.01+0.05/N,0.35}^T$. In the ADMM algorithm for SACA in Algorithm \ref{alg:SACA}, we initialize the penalty parameter $\rho$ by $\rho=1/\sqrt{N}$ and adaptively update it as in \cite[Section 3.4.1]{boyd2010distributed} to accelerate convergence. The ADMM will be terminated if the absolute and relative errors are below $10^{-4}$ and $10^{-5}$, respectively (see \cite[Section 3.3.1]{boyd2010distributed} for details), or a maximum number of $1000$ iterations are reached. Note that the basic code of SDPT3 is written in MATLAB, but key subroutines are implemented in Fortran and C and incorporated using Mex files \cite{toh2012implementation}. This integration has a positive impact on the CPU time for SACA with IPM.
The results are shown in \cref{tab:1}. 
ZF-ACMA is omitted as it exhibits a similar performance to ACMA. It is seen from \cref{tab:1} that SACA with ADMM is faster than SACA with IPM, particularly when $N=51$, although it has a slightly higher RMSE. ANM with ADMM, ACMA, and NM have low time costs, but their RMSE values are noticeably higher than those of SACA and SMART. 

\begin{table*}[htbp]
	\centering
	\caption{ Average RMSE and CPU time when $N=11$ and $N=51$ }
	\label{tab:1}
	\resizebox{0.6\columnwidth}{!}{
		\begin{tabular}{llcc}
			\toprule 
			$N$ & Method & RMSE & Time (s) \\
			\midrule 
			$11$ & SACA (IPM)  & 3.36e-03 & 4.58  \\
			{ } & SACA (ADMM)  & 4.18e-03 & 0.71  \\
			{ } & ANM (ADMM)  & 1.32e-01 & 0.24  \\
			{ } & SMART  & 1.46e-03 & 4.43  \\
			{ } & ACMA  & 6.85e-02 & 0.25  \\
			{ } & NM  & 6.40e-02 & 0.42  \\
			$51$ & SACA (IPM)  & 1.50e-04 & 118.68  \\
			{ } & SACA (ADMM)  & 4.84e-04 & 11.43  \\
			{ } & ANM (ADMM)  & 5.71e-02 & 1.26 \\
			{ } & SMART  & 2.21e-04 & 39.06  \\
			{ } & ACMA  & 2.93e-02 & 0.35  \\
			{ } & NM  & 2.92e-02 & 0.49  \\
			\bottomrule
		\end{tabular}
	}
\end{table*}

\section{Conclusion}
In this paper, the convex approach SACA was proposed for the frequency estimation problem with constant amplitude based on atomic norm minimization. A novel convex structured low-rank approximation characterization of SACA was formulated based on multiple Hankel-Toeplitz matrices. The advantage of SACA due to the use of the constant amplitude property was analyzed theoretically. Simulation results were provided that validate our theoretical findings and demonstrate superior performance of SACA as compared to existing methods.

Since the traditional IPM is a second-order method and the ADMM algorithm for SACA needs to compute eigen-decompositions in each iteration, both of them become slow for large-scale CAFE problems. A future research direction is to develop faster solvers for the proposed SLRA problems of SACA, e.g., by incorporating the Burer-Monteiro factorization technique \cite{burer2003nonlinear} into the IPM \cite{bellavia2021relaxed}, enabling operations on smaller factor matrices rather than full-rank matrices, or the ADMM algorithm, avoiding the eigen-decompositions.

\appendix

\section{Proof of \cref{lem:twonorm}} \label{append:1}
According to the definition of atomic norm, it suffices to show that the convex hulls of $\cA$ and $\cA'$ are equal, i.e., 
$
{\rm conv}\sbra{\cA} = {\rm conv}\sbra{\cA'}.
$
We have $\text{conv} \sbra{\cA} \subset \text{conv} \sbra{\cA'}$ since $\cA \subset \cA'$. On the other hand, suppose that $\m{a}(f) \m{\psi} \in \cA'$. Since $\norm{\m{\psi}}_{\infty} = 1$, 
we have $|\psi_l| \leq 1, l=1,\dots,L$. For any $\psi_l$ in the unit circle, we can easily find $\phi_{l,1}, \phi_{l,2}$ on the unit circle such that 
$
\psi_l = \frac{1}{2} \phi_{l,1} + \frac{1}{2} \phi_{l,2}.
$
When $|\psi_l| = 1$, we take $\phi_{l,1}=\phi_{l,2} = \psi_l$.
Consequently, we can always find unit-modulus vectors $\m{\phi}_1$ and $\m{\phi}_2$, formed by ${\phi_{l,1}}$ and ${\phi_{l,2}}$ respectively, such that
$\m{\psi}_l = \frac{1}{2} \m{\phi}_1 + \frac{1}{2} \m{\phi}_2$.
It follows that
$
\m{a}(f)\m{\psi} = \frac{1}{2} \m{a}(f)\m{\phi}_1 + \frac{1}{2} \m{a}(f)\m{\phi}_2,
$
implying that $\cA' \subset \text{conv} \sbra{\cA}$ since $\m{a}(f)\m{\phi}_1, \m{a}(f)\m{\phi}_2 \in \mathcal{A}$. Then,  $\text{conv} \sbra{\cA'} \subset \text{conv} \sbra{\cA}$. Therefore, we conclude that $\text{conv} \sbra{\cA} = \text{conv} \sbra{\cA'}$, completing the proof.

\section{Derivation of \eqref{eq:dualSDP}} \label{append:dual}
Following from a standard Lagrangian analysis \cite{boyd2004convex}, we introduce a multiplier $\m{F}$ for the equality constraint and multiple PSD Lagrangian multipliers $\lbra{\m{W}^l \succeq \m{0}}^L_{l=1}$ for the inequality constraints. Write $\m{W}^l=\begin{bmatrix} \m{W}^{l,1} & \sbra{\m{U}^l}^H \\ \m{U}^l & \m{W}^{l,2} \end{bmatrix}$. Then the Lagrangian is given by:
\equ{
	\begin{split}
		& \cL' \sbra{\m{t},\m{Z},\m{F},\lbra{\m{W}^l}}
		\\
		& = t_n  - \sum^L_{l=1} \left \langle \begin{bmatrix} \m{W}^{l,1} & \sbra{\m{U}^l}^H \\ \m{U}^l & \m{W}^{l,2} \end{bmatrix}, \begin{bmatrix} \cT\overline{\m{t}} & \cH\overline{\m{Z}}_{:,l} \\ \cH\m{Z}_{:,l} & \cT\m{t} \end{bmatrix}  \right \rangle_{\bR}  + \left \langle \m{F}, \m{Z}_{ \Omega} - \m{X}^{\star}_{\Omega} \right \rangle_{\bR} \\
		& =  t_n -  \left \langle \sum^L_{l=1}\sbra{\overline{\m{W}^{l,1}}+\m{W}^{l,2}}, \cT\m{t} \right \rangle_{\bR}  - 2 \sum^L_{l=1} \left \langle \m{U}^l, \cH\m{Z}_{:,l} \right \rangle_{\bR}  + \left \langle \m{F}, \m{Z}_{ \Omega} \right \rangle_{\bR} - \left \langle \m{F}, \m{X}^{\star}_{\Omega} \right \rangle_{\bR} . \label{eq:Lag}
	\end{split}
}
The dual function is given by:
$
g\sbra{\m{F},\lbra{\m{W}^l}} = \inf_{\m{t},\m{Z}} \cL' \sbra{\m{t},\m{Z},\m{F},\lbra{\m{W}^l}}.
$
Then we consider the optimality conditions
\equ{
	\begin{split}
		\nabla_{\m{t}} \cL'\sbra{\m{t},\m{Z},\lbra{\m{W}^l}} &= \m{0}, \\
		\nabla_{\m{Z}_{:,l}} \cL'\sbra{\m{t},\m{Z},\lbra{\m{W}^l}} &= \m{0}, \ l=1,\ldots, L,
	\end{split}
}
which yields 
\equ{
	\begin{split}
		\cT^H\lbra{ \sum^L_{l=1}\sbra{\overline{\m{W}^{l,1}}+\m{W}^{l,2}} } &= \m{\xi}, \\
		\m{F}_{:,l} & = 2 \sbra{\cH^H\m{U}^l}_{\Omega}, \\
		\sbra{\cH^H\m{U}^l}_{\lbra{1,\ldots,2n-1} - \Omega} &= \m{0}, \ l = 1,\ldots,L. \label{eq:opt_cond}
	\end{split} 
}
Putting \eqref{eq:opt_cond} into \eqref{eq:Lag} and letting $\m{V}_{:,l} = -2 \sbra{\cH^H\m{U}^l}_{\lbra{1,\ldots,N}}$, 
we get the dual problem $\max g\sbra{\lbra{\m{W}^l}}$ in \eqref{eq:dualSDP}. 

\section{Proof of \cref{lem:1}} \label{append:thm}
Note that the bounded random variables $\lbra{\Phi_{k,l} = e^{i\phi_{k,l}}}$ are independent sub-Gaussian with norm $1/\sqrt{\ln 2}$.
According to \cite[Proposition 2.6.1]{vershynin2018high}, the entries of $\m{\Phi}^H \m{w}$ are independent sub-Gaussian with norm $C\twon{\m{w}}$, where $C$ is an absolute constant. Besides, for $l=1,\ldots,L$, $\bE \abs{\m{\Phi}_{:,l}^H\m{w}}^2 = \bE \m{w}^H \m{\Phi}_{:,l}\m{\Phi}_{:,l}^H \m{w} = \twon{\m{w}}^2 $. Applying \cite[Theorem 3.1.1]{vershynin2018high}, we have for all $t\ge 0$ and a numerical constant $c$ that
$
\bP \lbra{ \twon{\frac{\m{\Phi}^H\m{w}}{\twon{\m{w}}}} \ge \sqrt{L}+t } \le e^{-ct^2},
$
or equivalently,
\equ{
	\bP \lbra{\twon{\frac{1}{\sqrt{L}}\m{\Phi}^H\m{w}} \ge \sbra{1+\frac{t}{\sqrt{L}}} \twon{\m{w}} } \le e^{-ct^2}.\nonumber
}
Letting $u=\sbra{1+\frac{t}{\sqrt{L}}} \twon{\m{w}}$, we complete the proof.

\section{Proof of \cref{thm:noisy}} \label{append:noisy}
According to \eqref{eq:dualnorm}, the dual atomic norm of $\m{E}_{\Omega} \in \bC^{M \times L}$ is given by:
\equ{
	\norm{\m{E}_{\Omega}}^*_{\cA_{\Omega}} = \sup_{f\in \bT,\m{u}\in \mathcal{B}_{\infty}^L(1)} \abs{\m{u}^H\m{E}_{\Omega}^H \m{a}_{\Omega}(f)},
}
where $\mathcal{B}_{\infty}^L(1) = \lbra{\m{x}\in \bC^L: \norm{\m{x}}_{\infty}= 1 }$. Let $\cN_1 = \lbra{-\frac{1}{2},-\frac{1}{2}+2\varepsilon_1, \ldots, \frac{1}{2}-2\varepsilon_1}$ with $\varepsilon_1\in (0,1]$ that is an $\varepsilon_1$-net of the frequency interval $\bT$ in the sense that for any $f \in \bT$, we can always find $f_0\in \cN_1$ such that $\abs{f-f_0}\le \varepsilon_1$. Evidently, we have the cardinality $\abs{\cN_1} = 1/(2\varepsilon_1)$. Moreover, let $\cN_2$ be an $\varepsilon_2$-net of $\mathcal{B}_{\infty}^L(1)$ with $\varepsilon_2\in (0,1]$ regarding the $\ell_{\infty}$ metric. In particular, for any $\m{u}\in \mathcal{B}_{\infty}^L(1)$, we can always find $\m{u}_0 \in \cN_2$ such that $\norm{\m{u} - \m{u}_0}_{\infty} \leq \varepsilon_2$. Using \cite[Corollary 4.2.13]{vershynin2018high}, we obtain that $\abs{\cN_2} \leq (3/\varepsilon_2)^{2L}$. 
As in \cite{bhaskar2013atomic}, our proof is based on the following two results.
\begin{lemma} [\cite{schaeffer1941inequalities}] \label{lem:derivative}
	Let $q\sbra{z}$ be any polynomial of degree $N$ on complex numbers with derivative $q'\sbra{z}$. Then, $\sup_{\abs{z}\le 1} \abs{q'\sbra{z}} \le N \sup_{\abs{z}\le 1} \abs{q\sbra{z}} $.
\end{lemma}

\begin{lemma}[\cite{bhaskar2013atomic}] \label{lem:Gaussian}
	Let $x_1,\ldots,x_N$ be complex Gaussian random variables with unit variance. Then, $\bE \mbra{\max_{1\le j\le N} \abs{x_j}} \le \sqrt{\ln N +1}$.
\end{lemma}

Denote $W\sbra{e^{i2\pi f}} = \m{u}^H\m{E}_{\Omega}^H \m{a}_{\Omega}(f)$. According to \cref{lem:derivative}, for any $f_0,f \in \bT$, we have
\equ{
	\begin{split}
		\abs{ \m{u}^H\m{E}_{\Omega}^H \m{a}_{\Omega}(f) } - \abs{ \m{u}^H\m{E}_{\Omega}^H\m{a}_{\Omega}(f_0)) } 
		& \le \abs{e^{i2\pi f}-e^{i2\pi f_0}} \cdot \sup_f \abs{W'} \\
		& = \abs{e^{i\pi \sbra{f+f_0}}\sbra{e^{i\pi \sbra{f-f_0}} -e^{i\pi\sbra{-f+f_0}} }} \cdot \sup_f \abs{W'} \\
		& \le 2\pi \abs{f-f_0} \cdot \overline{N} \cdot \sup_f \abs{W}. \label{eq:f0}
	\end{split}
}
It then follows from \eqref{eq:f0} that
\equ{
	\begin{split}
		& \abs{\m{u}^H\m{E}_{\Omega}^H \m{a}_{\Omega}(f)} - \abs{\m{u}_0^H\m{E}_{\Omega}^H \m{a}_{\Omega}(f_0)} 
		\\
		& \le \abs{ \m{u}^H\m{E}_{\Omega}^H \m{a}_{\Omega}(f) } - \abs{ \m{u}^H\m{E}_{\Omega}^H\m{a}_{\Omega}(f_0) } + 
		\abs{ (\m{u}-\m{u}_0)^H\m{E}_{\Omega}^H \m{a}_{\Omega}(f_0) } \\
		& \le 2\pi \abs{f-f_0} \overline{N} \sup_f \abs{\m{u}^H\m{E}_{\Omega}^H\m{a}_{\Omega}(f)}  + \norm{\m{u}-\m{u}_0}_{\infty} \norm{\m{E}_{\Omega}^H\m{a}_{\Omega}(f_0)}_1 \\
		& \le 2\pi \overline{N} \varepsilon_1 \sup_f \abs{\m{u}^H\m{E}_{\Omega}^H\m{a}_{\Omega}(f)} + 2\pi \varepsilon_2 \norm{\m{E}_{\Omega}^H\m{a}_{\Omega}(f_0)}_1. \nonumber
	\end{split}
}
Consequently, we have
\equ{
	\begin{split}
		\bE \norm{\m{E}_{\Omega}}^*_{\cA_{\Omega}} & \le \sbra{ 1-2\pi \overline{N} \varepsilon_1 - 2\pi \varepsilon_2 }^{-1} \bE \mbra{ \sup_{f_0\in \cN_1, \m{u}_0\in \cN_2 } \abs{ \m{u}_0^H \m{E}_{\Omega}^H \m{a}_{\Omega}(f_0) } } \\
		& \le \sbra{ 1-2\pi \overline{N} \varepsilon_1 - 2\pi\varepsilon_2 }^{-1} \sqrt{ \sigma M L } \sqrt{ \ln{ \frac{1}{2\varepsilon_1} \sbra{\frac{3}{\varepsilon_2}}^{2L} } + 1 }, \nonumber 
	\end{split}
}
where the second inequality follows from \cref{lem:Gaussian} and the fact that $\m{u}_0^H \m{E}_{\Omega}^H \m{a}_{\Omega}(f_0)$ is Gaussian with zero mean and variance
\equ{
	\begin{split}
		\bE \abs{\m{u}_0^H \m{E}_{\Omega}^H \m{a}_{\Omega}(f_0)}^2  
		& = \bE \abs{ \langle \text{vec} \sbra{\m{E}_{\Omega}}, \text{vec} \sbra{\m{a}_{\Omega}(f_0) \m{u}_0^H} \rangle }^2  = \sigma \twon{\text{vec} \sbra{\m{a}_{\Omega}(f_0) \m{u}_0^H}}^2 \\
		& = \sigma \twon{\m{a}_{\Omega}(f_0)}^2 \twon{\m{u}_0}^2 = \sigma M L. \nonumber
	\end{split}
}
Further, let $p_1=1/\sbra{2\pi \overline{N} \varepsilon_1}, p_2 = 1/\sbra{2\pi \varepsilon_2}$, we obtain
\equ{
	\begin{split}
		\bE \norm{\m{E}_{\Omega}}^*_{\cA_{\Omega}}  \le  \sbra{1-\frac{1}{p_1}-\frac{1}{p_2}}^{-1}  \sqrt{\sigma M L \sbra{2L \ln (6\pi p_2 ) + \ln \overline{N} + \ln \pi p_1 + 1} }. \nonumber
\end{split}	}
It can be shown that the right hand side is minimized approximately when $p_1 = 4\ln \sbra{L \ln 8\pi + \ln \overline{N}}$ and $p_2 = 4$. 
Hence, the upper bound on $\bE \norm{\m{E}_{\Omega}}^*_{\cA_{\Omega}}$ is of order $\sqrt{\sigma M L \sbra{L+\ln \overline{N}}}$.

\bibliographystyle{siamplain}
\bibliography{SACA}
\end{document}